\documentclass[12pt,twoside,reqno]{amsart}

\usepackage{amssymb,amsmath,amstext,amsthm,amsfonts,amscd,xcolor}

\usepackage{dsfont}

\usepackage[ansinew]{inputenc} 
\usepackage{graphicx}
\usepackage[mathscr]{eucal}
\usepackage{hyperref}

\newcommand{\blob}{\rule[.2ex]{.8ex}{.8ex}}

\newcommand{\R}{\mathbb{R}}
\newcommand{\C}{\mathbb{C}}
\newcommand{\N}{\mathbb{N}}
\newcommand{\Z}{\mathbb{Z}}

\newcommand{\GL}{{\rm GL}}

\newcommand{\Mat}{{\rm Mat}}
\newcommand{\V}{\mathscr{V}}
\newcommand{\EE}{\mathbb{E}}

\newcommand{\abs}[1]{\bigl| #1 \bigr|} % absolute value
\newcommand{\norm}[1]{\lVert#1\rVert} % norm
\newcommand{\transl}{{T}} % base dynamics map

\newcommand{\ep}{\epsilon}

\newcommand{\cocycles}{\mathcal{C}}
\newcommand{\observables}{\Xi}

\newcommand{\params}{\mathscr{P}}

\newcommand{\mesfs}{\mathscr{I}}
\newcommand{\devfs}{\mathscr{E}}

\newcommand{\dev}{\ep}                %%% Size of deviation
\newcommand{\devf}{\underline{\dev}}  %%% Deviation function 
\newcommand{\mes}{\iota}              %%% Measure of deviation set in LDT
\newcommand{\mesf}{\underline{\mes}}  %%%   Function measuring deviation set in LDT

\newcommand{\An}[1]{A^{({#1})}}  %%% nth iteration of cocycle A
\newcommand{\Gr}{{\rm Gr}}
\newcommand{\Pp}{\mathbb{P}}
\newcommand{\FF}{\mathscr{F}}

\newcommand{\ind}{\mathds{1}}

\newcommand{\B}{\mathscr{B}}

\newcommand{\dist}{{\rm dist}}
\newcommand{\one}{{\bf 1}}
\newcommand{\Ktil}{\widetilde{K}}

\newcommand{\dtil}{\widetilde{d}}

\newcommand{\Proj}{\mathbb{P}(\R^{m})}
\newcommand{\Ahat}{A}
\newcommand{\Hscr}{\mathscr{H}}

\newcommand{\filt}{F}

\newcommand{\Filt}{\mathfrak{F}}

\newcommand{\Filatp}[1]{\mathfrak{F}_{\supset{#1}}}

\newcommand{\dec}{E_{\cdot}}
\newcommand{\Dec}{\mathfrak{D}}

\newcommand{\Decatp}[1]{\mathfrak{D}_{\supset{#1}}}

\makeatletter
\newsavebox{\@brx}
\newcommand{\llangle}[1][]{\savebox{\@brx}{\(\m@th{#1\langle}\)}%
  \mathopen{\copy\@brx\mkern2mu\kern-0.9\wd\@brx\usebox{\@brx}}}
\newcommand{\rrangle}[1][]{\savebox{\@brx}{\(\m@th{#1\rangle}\)}%
  \mathclose{\copy\@brx\mkern2mu\kern-0.9\wd\@brx\usebox{\@brx}}}
\makeatother

\newcommand{\Fcal}{\mathcal{F}}
\newcommand{\randcocycles}[2]{\B^{{#1}}_{{#2}}(K)}
\newcommand{\prandcocycles}[2]{\B^{{#1}}_{{#2}}(K,p)}
\newcommand{\irredcocycles}[2]{\mathcal{I}^{{#1}}_{{#2}}(K)}
\newcommand{\pirredcocycles}[2]{\mathcal{I}^{{#1}}_{{#2}}(K,p)}

\newcommand{\ctwo}{h}
\newcommand{\Xfrk}{\mathfrak{X}}

\newcommand{\normtwo}[1]{% Peter Grill norm @tex.stackexchange.com
{\left\vert\kern-0.25ex\left\vert\kern-0.25ex\left\vert #1 
    \right\vert\kern-0.25ex\right\vert\kern-0.25ex\right\vert} } 
\newcommand{\x}{(x_1, x_2)}

\theoremstyle{plain}
\newtheorem{theorem}{Theorem}[section]
\newtheorem{proposition}{Proposition}[section]
\newtheorem{corollary}[proposition]{Corollary}
\newtheorem{lemma}[proposition]{Lemma}
\newtheorem{definition}{Definition}[section]

\newtheorem{remark}{Remark}[section]

\title[Large deviations for random cocycles]{Large deviation type estimates for random cocycles}

\date{}

\begin{document}

\author[P. Duarte]{Pedro Duarte}
\address{Departamento de Matem\'atica and CMAFIO\\
Faculdade de Ci\^encias\\
Universidade de Lisboa\\
Portugal 
}
\email{pmduarte@fc.ul.pt}

\author[S. Klein]{Silvius Klein}
\address{ Department of Mathematical Sciences\\
Norwegian University of Science and Technology (NTNU)\\
Trondheim, Norway\\ 
and IMAR, Bucharest, Romania }
\email{silvius.klein@math.ntnu.no}

\begin{abstract}
In this paper we prove the continuity of all Lyapunov exponents, as well as the continuity of the Oseledets decomposition, for a class of irreducible cocycles over strongly mixing Markov shifts. 
Moreover, gaps in the Lyapunov spectrum lead to a H\"older modulus of  continuity for these quantities.
This result is an application of the abstract continuity theorems  obtained in~\cite{LEbook}, and generalizes a theorem of
E. Le Page on the H\"older continuity of the maximal LE
for one-parameter families of strongly irreducible and contracting cocycles over a Bernoulli shift.

This is a draft of a chapter  in our forthcoming research monograph~\cite{LEbook}.  %[4]
\end{abstract}

\maketitle

\section{Introduction and statements}
\label{random_intro}

We define the class of random cocycles over Markov shifts  and describe our assumptions on them. We then formulate the main statements, and sketch the argument for proving large deviation type estimates. Finally we relate our findings to other results for similar models.

\subsection{Description of the model}
\label{random_model}
Let $\Sigma$ be a compact metric space and $\FF$  its Borel $\sigma$-field.
 
\begin{definition}\label{Markov:kernel}
A {\em Markov kernel} is a function
$K:\Sigma\times\FF\to [0,1]$ such that
\begin{enumerate}
\item[(1)] for every $x\in \Sigma$, $A\mapsto K(x,A)$ is a probability measure in $\Sigma$, also denoted by $K_x$,
\item[(2)] for every $A\in\FF$, the function $x\mapsto K(x,A)$ is $\FF$-measurable.
\end{enumerate}
\end{definition}
The iterated Markov kernels are defined  recursively, setting  
\begin{enumerate}
\item[(a)] $K^1=K$,
\item[(b)]  $K^{n+1}(x,A)= \int_\Sigma K^n(y,A)\, K(x,dy)$, \,  for all $n\geq 1$.
\end{enumerate}
Each power $K^n$ is itself a  Markov kernel on $(\Sigma,\FF)$.

%% stationary measures
A probability measure $\mu$ on $(\Sigma,\FF)$ is called {\em $K$-stationary} if for all $A\in\FF$,
$$ \mu(A)=\int  K(x,A)\, \mu(dx)\;.  $$
%ergodicity
A set $A\in\FF$ is said to be {\em $K$-invariant} when
$K(x,A)=1$ for all $x\in A$ and $K(x,A)=0$ for all
$x\in X\setminus A$. A $K$-stationary measure $\mu$ is called {\em ergodic} when there is no $K$-invariant set $A\in \FF$ such that $0<\mu(A)<1$. As usual, ergodic measures are the extremal points in the convex set of $K$-stationary measures.

%%Markov systems
\begin{definition}
A  {\em Markov system} is a pair $(K,\mu)$, where $K$ is a Markov kernel on $(\Sigma,\FF)$ and  $\mu$ is a $K$-stationary probability measure.
\end{definition}

Let $(K,\mu)$ be a Markov system.
There is a canonical construction, due to Kolmogorov,
of a probability space $(X,\FF,\Pp_\mu)$ and a Markov stochastic process $\{e_n:X\to \Sigma\}_{n\geq 0}$ with initial distribution $\mu$ and transition kernel $K$, i.e., for all $x\in\Sigma$ and $A\in\FF$,
\begin{enumerate}
\item $\Pp_\mu[ \, e_0\in A\,] = \mu(A)$, 
\item $\Pp_\mu[ \, e_{n}\in A\,\vert \, e_{n-1}=x\, ] = K(x, A)$.
\end{enumerate}

We briefly outline this construction. Elements in $\Sigma$ are called {\em states}.
Consider the space $X^+=\Sigma^\N$  of {\em state sequences}
 $x=(x_n)_{n\in\N}$, with $x_n\in\Sigma$ for all $n\in\N$,  and let $\Fcal^+$ be the product $\sigma$-field
$\Fcal^+=\FF^ \N$ generated by the $\FF$-cylinders, i.e.,  generated  by sets of the form
$$ C(A_0,\ldots, A_m):=\{\, x\in X^+\,\colon\,
x_{j}\in A_j,\; \text{ for }\, 0\leq j\leq m\, \} \;, $$
where  $A_0,\ldots, A_m\in\FF$ are measurable sets.
The (topological) product space $X^ +$ is compact and metrizable.
The $\sigma$-field  $\Fcal^+$ coincides with the Borel $\sigma$-field of the compact space $X^ +$.
\begin{definition}
\label{def measure Pp theta}
Given any probability measure $\theta$ on $(\Sigma,\FF)$, the following expression determines a pre-measure
$$ \Pp_\theta^+[ C(A_0,\ldots, A_m) ]:=
\int_{A_m} \cdots \int_{A_0} \theta(dx_0)\,\prod_{j=1}^{m} K(x_{j-1},dx_j)$$
on the  semi-algebra of $\FF$-cylinders. By Carath\'eodory's extension theorem this pre-measure extends to a unique probability measure
$\Pp_\theta^+$ on $(X^+,\Fcal^+)$.
\end{definition}
It follows from this definition that the 
sequence of random variables $e_n:X^+\to\Sigma$, defined by $e_n(x):=x_n$  for $x=(x_n)_{n\in\N}$,
is a Markov chain with initial distribution $\theta$ and transition kernel $K$ w.r.t. the probability space $(X^+,\Fcal^+,\Pp_\theta^+)$. 
It also follows that  the process $\{e_n\}_{n\geq 0}$ is stationary w.r.t. $(X,\Fcal^+,\Pp_\theta^+)$ if and only if   $\theta$ is a $K$-stationary measure.

Consider now the space $X=\Sigma^\Z$ of bi-infinite {\em state sequences}
 $x=(x_n)_{n\in\Z}$, with $x_n\in\Sigma$ for all $n\in\Z$,  and let $\Fcal$ be the product $\sigma$-field
$\Fcal=\FF^ \Z$ generated by the $\FF$-cylinders in $X$.
Again the  topological product space $X$ is both metrizable and compact, and the $\sigma$-field  $\Fcal$ is the Borel $\sigma$-field on the compact metric space $X$.
There is a canonical projection  
 $\pi:X\to X^+$, defined by $\pi(x_n)_{n\in\Z}=(x_n)_{n\in\N}$, relating these two spaces.

Markov systems are probabilistic evolutionary models, which can also be studied in
dynamical terms. For that we introduce  the {\em shift mappings}. 

\begin{definition}
The {\em one-sided shift} is the map  $T:X^+\to X^+$, \, $T(x_n)_{n\geq 0}=(x_{n+1})_{n\geq 0}$,
while the {\em two-sided shift} is the map  $T:X\to X$, \, $T(x_n)_{n\in\Z}=(x_{n+1})_{n\in\Z}$.
\end{definition}

The map $T:X^+\to X^+$ is continuous, and hence $\Fcal^+$-measurable.
It also preserves the measure $\Pp_\mu^+$, i.e., $T_\ast\Pp_\mu^+=\Pp_\mu^+$.
Moreover, the Markov process $\{e_n\}_{n\geq 0}$ on $(X^+,\Fcal^+,\Pp_\mu^+)$ is dynamically generated by the observable $e_0$ in the sense that
$e_n= e_0\circ T^n$, for all $n\geq 0$.

The two-sided-shift  $T:X\to X$ 
is a homeomorphism, and hence $\Fcal$-bimeasurable.
The projection $\pi:X\to X^+$ semi-conjugates the two shifts. The two-sided-shift  is the {\em natural extension} of the one-sided-shift.
According to this construction (see~\cite{Petersen}), there is a unique probability measure $\Pp_\mu$ on $(X,\Fcal)$ such that  $T_\ast\Pp_\mu=\Pp_\mu$ and  $\pi_\ast\Pp_\mu=\Pp_\mu^+$. We will refer to the measures $\Pp_\mu^+$ and $\Pp_\mu$  as the {\em Kolmogorov extensions} of the Markov system $(K,\mu)$.

\begin{definition}\label{Markov:dyn:sys}
Given a Markov system  $(K,\mu)$ let
$\Pp_\mu$ be the Kolmogorov extension of $(K,\mu)$ on $X=\Sigma^ \Z$. 
The dynamical system  $(X, \Pp_\mu, T)$ is called a {\em Markov shift}.
\end{definition}

%\bigskip

Let $\left(L^\infty (\Sigma),\norm{\cdot}_\infty\right)$
denote the Banach algebra of complex bounded $\FF$-measurable functions with the sup norm $\norm{f}_\infty=\sup_{x\in \Sigma} \abs{f(x)}$.
The following concept corresponds to condition (A1) in~\cite{Bou}.

\begin{definition} 
\label{Strong:mixing}
We say that a Markov system $(K,\mu)$ is  
{\em strongly mixing } if  there are constants
$C>0$ and $0<\rho<1$ such that for every $f\in L^\infty(\Sigma)$, all $x\in\Sigma$ and $n\in\N$,  
\begin{equation*}
\abs{\int_\Sigma f(y)\,K^n(x,dy) - \int_\Sigma f(y)\,\mu(dy)  }  \leq C\,\rho^n \,\norm{f}_\infty   \;. 
\end{equation*}
\end{definition}

It follows from this definition that,

\begin{proposition}
If the Markov system  $(K,\mu)$ is  strongly mixing  then the Markov shift $(X,\Pp_\mu,T)$ is a mixing dynamical system. 
\end{proposition}

\begin{proof}
Consider a bounded measurable observable
$f:X\to\R$ depending only on the coordinates
$x_0,\ldots, x_p$, and write $f(x)=f(x_0,\ldots, x_p)$. Let $g(x)=g(x_{-q},\ldots, x_{-1})$ be another bounded measurable observable depending  only on the coordinates
$x_{-q},\ldots, x_{-1}$ with $q\in\N$. 
Denote by $\{e_n\}_{n\in \Z}$ the Markov process
on $(X,\Pp_\mu)$ with common distribution $\mu$ and transition kernel $K$.
By the strong mixing property
\begin{align*}
\EE_{x_0} [ f(e_n,\ldots, e_{n+p}) ] = \int_\Sigma \cdots \int_\Sigma
 f(x_n,\ldots, x_{n+p})\,
K^{n}(x_0, d x_n)\,\prod_{j=n}^{n+p-1} K(x_{j}, d x_{j+1}) 
\end{align*}
converges uniformely (in $x_0$) to
\begin{align*}
\int_\Sigma \cdots \int_\Sigma
 f(x_n,\ldots, x_{n+p})\,
\mu(d x_n)\,\prod_{j=n}^{n+p-1} K(x_{j}, d x_{j+1}) =  \EE_\mu(f)\;.
\end{align*}
Hence
\begin{align*}
& \EE_\mu[ (f\circ T^n)\,g ] =
\EE_\mu [g(e_{-q},\ldots, e_{-1})\, f(e_n,\ldots, e_{n+p}) ] \\
&\quad = \int_\Sigma \cdots \int_\Sigma
 g(x_{-q},\ldots, x_{-1})\, \EE_{x_0}[ f(e_n,\ldots, e_{n+p}) ]\, \mu(d x_{-q})\,\prod_{j=-q}^{-1} K(x_{j}, d x_{j+1}) 
\end{align*}
converges to
\begin{align*}
& \EE_\mu(f)\,  \int_\Sigma \cdots \int_\Sigma
 g(x_{-q},\ldots, x_{-1})\,
\mu(d x_{-q})\,\prod_{j=-q}^{-1} K(x_{j}, d x_{j+1}) = \EE_\mu(f)\, \EE_\mu(g)\;.
\end{align*}
The mixing property of the shift $(X,\Pp_\mu,T)$ follows applying the previous argument to the indicator funtions of any two cylinders, and noting that the $\sigma$-algebra of cylinders generates the Borel $\sigma$-field of $X$.
\end{proof}

Examples of strongly mixing Markov systems arise naturally from Markov kernels
satisfying the  {\em Doeblin condition} (see~\cite{Doob}). 
 We say that $K$ satisfies the  Doeblin condition
 if  there is a positive finite measure $\rho$ on $(\Sigma,\FF)$ and some $\varepsilon>0$ such that for all $x\in\Sigma$ and $A\in\FF$, 
 $$ K(x,A)\geq 1-\varepsilon\quad \Rightarrow \quad  \rho(A)\geq\varepsilon\;. $$
Given $A\in \FF$,  define
$$L^\infty(A):=\{\, f\in L^\infty(\Sigma)\,:\,
f\vert_{\Sigma\setminus A}\equiv 0\,\}\;,$$
which is a closed Banach sub-algebra of
$\left(L^\infty(\Sigma),\norm{\cdot}_\infty\right)$.

\begin{proposition} Let $(\Sigma,K)$ be a Markov system.
If $K$ satisfies the Doeblin condition  then there are sets $\Sigma_1,\ldots, \Sigma_m$ in $\FF$ and  probability measures 
$\nu_1,\ldots, \nu_m$ on $\Sigma$ such that
for all $i,j=1,\ldots, m$,
\begin{enumerate}
\item $\Sigma_i\cap \Sigma_j=\emptyset$ \, when \,$i\neq j$,
\item $\Sigma_i$ is $K$-forward invariant,
i.e., $K(x,\Sigma_i)=1$ for  $x\in\Sigma_i$,
\item $\nu_i$ is  $K$-stationary and ergodic with $\nu_i(\Sigma_j)=\delta_{ij}$,
\item $lim_{n\to+\infty} K^{n}(x,\Sigma_1\cup \ldots \cup\Sigma_m)=1$, with geometric uniform speed of convergence,
for all $x\in\Sigma$,
\item $\nu(\Sigma_1\cup \ldots \cup\Sigma_m)=1$, for every $K$-stationary probability $\nu$.
\end{enumerate}
Moreover, for every $1\leq i\leq m$ there is an integer $p_i\in\N$ 
and measurable sets $\Sigma_{i,1},\ldots, \Sigma_{i,p_i}\in\FF$
such that
\begin{enumerate}
\item   $\{\Sigma_{i,1},\ldots, \Sigma_{i,p_i}\}$ is a  partition of  $\Sigma_i$,
\item $K(x,\Sigma_{i,j+1})=1$ for  $x\in \Sigma_{i,j}$  and $1\leq j \leq p_i$, with  $\Sigma_{i,p_i+1}=\Sigma_{i,1}$,
\item $(\Sigma_{i,j},K^{p_i})$ is strongly mixing 
for all $1\leq j \leq p_i$.
\end{enumerate}
\end{proposition}

\begin{proof}
See  ~\cite[section V-5]{Doob}.
\end{proof}

\medskip

%% Spaces of Cocycles
Let $(K,\mu)$ be a Markov system.
We introduce a space  of measurable
functions $A:\Sigma\times\Sigma \to \GL(m,\R)$.

\begin{definition} 
\label{random cocycles}
The space $\randcocycles{\infty}{m}$ consists of all  functions $A:\Sigma\times\Sigma \to \GL(m,\R)$ such that $A$ and $A^{-1}$ are both measurable and uniformly bounded. On this space we consider the metric 
$ d_\infty(A,B) = \norm{A-B}_\infty$.
\end{definition}

\begin{definition}
The function $A\in \randcocycles{\infty}{m}$
determines a  linear cocycle $F_A:X\times\R^ m\to X\times\R^m$ over the Markov shift  $(X,\Pp_\mu,T)$,   defined by
 $$ F_A(x, v) := (\transl x, \Ahat(x)\,v)\,,$$
 where we identify $A$ with the function $\Ahat:X\to \GL(m,\R)$, $\Ahat(x):=A(x_0,x_1)$,
 for  $x=(x_n)_{n\in\Z}\in X$.
\end{definition}
 
The iterates of $F_A$ are the maps
  $F_A^n:X\times\R^m\to X\times\R^m$,
 $$F_A^ n(x, v) = (\transl^n x, \Ahat^{(n)}(x)\,v)\;,$$
 with   $\Ahat^{(n)}:X\to \GL(m,\R)$   defined for all $x=(x_n)_{n\in\Z}$  by 
$$\Ahat^{(n)}(x):=A(x_{n-1},x_n) \ldots A(x_1,x_2)\,A(x_0,x_1)\;. $$

The cocycle $F_A$ is determined by the data
$(K,\mu,A)$, and identified by the function $A$, in contexts where the Markov system $(K,\mu)$ is fixed.

\begin{definition}
Let $\Gr(\R^m)$ denote the Grassmann manifold of the Euclidean space $\R^m$. A $\FF$-measurable section $V:\Sigma\to\Gr(\R^m)$
is called {\em $A$-invariant}   when  
$$ A(x_{n-1},x_n)\,V(x_{n-1})=V(x_n)\, \text{ for }\, \Pp_\mu\text{-a.e.}\; x=(x_n)_{n\in\Z}\;.$$
\end{definition}

Assuming $(K,\mu)$ is strongly mixing, the ergo\-di\-city of this Markov kernel  implies that the subspaces $V(x)$ have constant dimension $\mu$-a.e.,
 denoted by $\dim(V)$. We say that this family is {\em proper} if $0<\dim(V)<d$.

%% irreducibile cocycles
Next we introduce the concepts of irreducible and totally irreducible cocycle 
(see definition 2.7 in~\cite{Bou}).
\begin{definition}
A cocycle  $A\in\randcocycles{\infty}{m}$ is called  {\em irreducible}  w.r.t. $(K,\mu)$ if it admits no 
measurable proper $A$-invariant section $V:\Sigma\to \Gr(\R^m)$. A cocycle  $A\in\randcocycles{\infty}{m}$ is called  {\em totally irreducible}  w.r.t. $(K,\mu)$ if the exterior powers $\wedge_k A$ are irreducible for all $1\leq k\leq m-1$.
\end{definition}

We denote by $\irredcocycles{\infty}{m}$
the subspace of totally irreducible cocycles in
$\randcocycles{\infty}{m}$.

\begin{proposition}
\label{prop:irreducibility is open}
The subspace $\irredcocycles{\infty}{m}$ is open in $\randcocycles{\infty}{m}$.
\end{proposition}

\begin{proof}
A cocycle $A\in \randcocycles{\infty}{m}$ is reducible (i.e. not irreducible)  if it admits a measurable proper $A$-invariant section $V:\Sigma\to \Gr(\R^m)$.
It is enough to prove that the set of reducible cocycles is closed.

Let $A_k\to A$ be a convergent sequence
of reducible cocycles in $\randcocycles{\infty}{m}$, 
and  let $V_k:\Sigma\to \Gr(\R^m)$
be a  measurable  proper $A_k$-invariant section.
We will prove that $A$ is also reducible. 

We will assume the probability space $(\Sigma, \mu)$ to be complete.

Let $\Omega\subset X$ be a Borel measurable set with $\Pp_\mu(\Omega)=1$ such that for all $k\geq 1$  all $x=(x_n)_{n\in\Z}\in\Omega$ and $n\in\Z$,
\, $A_k(x_{n-1},x_n)\,V_k(x_{n-1})=V_k(x_n)$.

Fix any point $s_0\in\Sigma$. Extracting a subsequence we may assume that $V_k(s_0)$ converges to $V_0\in \Gr(\R^m)$  as $k$ tends to $\infty$.
Consider then the  set 
$$ A:=\{\, s\in\Sigma\,\colon\, \exists\, x\in\Omega, \, n\in\N\;
\text{ such that } \; x_0=s_0 \; \text{ and } \; x_n=s \, \} \;. $$
In general $A$ may fail to be a Borel set, but it is an {\em analytic set} in the sense of Descriptive set theory (see~\cite[Definition 14.1 and Exercise 14.3]{Kechris}). 
By~\cite[Theorem 21.10]{Kechris} this set is universally measurable, and in particular it is measurable w.r.t. $\mu$.
%Given $A\subset \Sigma$ measurable, by the strong mixing property\,
%$$\Pp_{s_0}\{ x\in\Omega\,\colon\, e_n(x)\in A\,\}= \EE_{s_0}[ \ind_{A}(e_n) ] \to \mu(A)\;.$$
Hence, because of  the strong mixing property,
\begin{align*}
\mu(\Sigma\setminus A) &= \lim_{n\to \infty} \EE_{s_0}[ \ind_{\Sigma\setminus A}(e_n) ] \\
& = \lim_{n\to \infty}
\Pp_{s_0}\{\, x\in\Omega\,\colon\,
e_n(x)\in\Sigma\setminus A\,\} =  \lim_{n\to \infty}
\Pp_{s_0}(\emptyset) =  0  \;,
\end{align*}
which proves that for $\mu$-a.e. $s\in\Sigma$ there exists 
a sequence $x\in \Omega$ such that $x_0=s_0$ and $x_n=s$ for some $n\in\N$.

Then  
$ V_k(s)= A_k(x_{n-1},x_n)\,\ldots \,
A_k(x_{1},x_2)\,A_k(x_{0},x_1)\, V_k(s_0)$, which  implies that $V_k(s)$ converges to $A(x_{n-1},x_n)\,\ldots \,
A(x_{1},x_2)\,A(x_{0},x_1)\, V_0$ when $k\to\infty$. Thus, $V_k(s)$ converges for $\mu$-a.e. $s\in\Sigma$, and the limit function $V(s)=\lim_{k\to \infty} V_k(s)$
is a measurable and proper $A$-invariant section,
with the same dimension as the sections $V_k$.
This proves that the cocycle $A$ is reducible.
\end{proof}

\medskip

For the reader's convenience we briefly recall some definitions and notations regarding the Lyapunov exponents, Oseledets filtrations and decompositions of a cocycle $A$ in any space of cocycles $\cocycles_m$.

The ergodic theorem of Kingman allows us to define the Lyapunov exponents $L_j(A)$ with $1\leq j\leq m$ as
$L_j(A):=\Lambda_j(A)-\Lambda_{j-1}(A)$ where
$$ \Lambda_j(A) := \lim_{n\to \infty} \frac{1}{n}\,\log \norm{\wedge_j A(x)} \quad \text{ for }\, \mu\text{-a.e. } x\in X\;. $$

Let $\tau=(1\leq \tau_1 <\ldots <\tau_k < m)$ be a signature. If  $A\in\cocycles_m$ has a $\tau$-gap pattern, i.e., $L_{\tau_j}(A)>L_{\tau_{j+1}}(A)$ for all $j$,
we define the Lyapunov $\tau$-block 
$$ \Lambda^\tau(A):= ( \Lambda_{\tau_1}(A),\ldots, \Lambda_{\tau_k}(A))\in\R^k\;.$$

A flag of $\R^m$ is any increasing sequence of linear subspaces. The corresponding sequence of dimensions is called its signature.
A measurable filtration is a measurable function
on $X$, taking values in the space of flags of $\R^m$ 
with almost sure  constant signature.
We denote by
$\Filt(X,\R^m)$ the space of measurable filtrations.
Note that the Oseledets filtration of $A$, which we denote by $\filt(A)$, is an element of this space.

We denote by  $\Filatp{\tau}(X,\R^m)$
the subset of measurable filtrations with a signature $\tau$ or finer. If $F\in \Filatp{\tau}(X,\R^m)$ there is a natural projection $F^\tau$ with signature $\tau$, obtained from $F$ by simply `forgetting' some of its components. This space is endowed with the following pseudo-metric
$$ \dist_\tau(F, F'):=
\int_X d_\tau( F^\tau(x), (F')^\tau(x) )\,  \mu(dx) \;,$$
where $d_\tau$ refers to the metric on the $\tau$-flag manifold.

On the space $\Filt(X,\R^m)$ we consider the coarsest topology that makes the sets $\Filatp{\tau}(X,\R^m)$ open, and the pseudo-metrics $\dist_\tau$ continuous.

%%%%%%%%%%%%%%%%%%

A decomposition of  $\R^m$ is a sequence of linear subspaces $\{E_j\}_{1\leq j \leq k+1}$ 
  whose direct sum is $\R^m$. This determines the flag 
  $E_1 \subset E_1\oplus E_2\subset \ldots  \subset E_1\oplus\ldots \oplus E_k$, whose signature $\tau$ also designates the signature of the decomposition.

A measurable decomposition is a measurable function
on $X$, taking values in the space of decompositions of $\R^m$ 
with almost sure  constant signature.
We denote by
$\Dec(X,\R^m)$  the space of measurable decompositions.
Note that the Oseledets decomposition of $A$, which we denote by $\dec(A)$, is an element of this space.

We denote by   $\Decatp{\tau}(X,\R^m)$
the subset of measurable decompositions with a signature $\tau$ or finer. If $\dec \in  \Decatp{\tau}(X,\R^m)$ there is a natural restriction $\dec^\tau$ with signature $\tau$, obtained from $\dec$ by simply `patching up' the appropriate components. This space is endowed with the following pseudo-metric
$$ \dist_\tau(\dec, \dec'):=
\int_X d_\tau( \dec^\tau(x), (\dec')^\tau(x) )\,  \mu(dx) \;,$$
where $d_\tau$ refers to the metric on the manifold of $\tau$-decompositions.

On the space $\Dec(X,\R^m)$  we consider the coarsest topology that makes the sets $\Decatp{\tau}(X,\R^m)$ open, and the pseudo-metrics $\dist_\tau$ continuous.

\medskip

We are ready to state a general result on the continuity of the LE, the Oseledets filtration and the Oseledets decomposition for irreducible Markov cocycles.

%Continuity theorems for Random systems

\begin{theorem}
\label{theorem: continuity of LE}
Let  $(K,\mu)$ be a  strongly mixing  Markov system
and let $m\geq 1$.

Then all Lyapunov exponents $L_j:\irredcocycles{\infty}{m}\to \R$, with $1\leq j \leq m$, 
the Oseledets filtration $\filt:\irredcocycles{\infty}{m}\to \Filt(X,\R^m)$,
and the Oseledets decomposition  $\dec:\irredcocycles{\infty}{m}\to \Dec(X,\R^m)$,
are continuous functions of the cocycle
$A\in \irredcocycles{\infty}{m}$. 

Moreover,  if $A\in \irredcocycles{\infty}{m}$ has a $\tau$-gap pattern then the functions
$\Lambda^\tau$, $\filt^\tau$ and $\dec^\tau$ are
H\"older continuous in a neighborhood of $A$.
\end{theorem}

This theorem is proved in section~\ref{random_cont_proof}. 
It is an application of theorems 3.1, 4.7 and 4.8  in~\cite{LEbook}.
The main ingredients in these applications are two theorems on base and fiber uniform LDT estimates  of exponential type that we now formulate.

We begin with the base LDT theorem.
Consider the  metric  
$\dtil:X\times X\to [0,1]$
\begin{equation}
\label{dtil}
 \dtil(x,x')
:= 2^{- \inf\{ \, \vert k\vert \,:\, k\in\Z,\,  x_{k}\neq x_{k}'\,\} }\;, 
\end{equation}
for all $x=(x_k)_{k\in\Z}$ and  $x'=(x_k')_{k\in\Z}$
in $X$.
Notice that  $X$ is not compact for the topology induced by
$\dtil$, unless  $\Sigma$ is finite.
Given $k\in\N$, $\alpha > 0$  and $f\in L^\infty(X)$ define
\begin{align*}
v_k(f) &:=\sup\{\, \abs{f(x)-f(y)}\,:\, \dtil(x,y)\leq 2^{-k}\,\} \;, \\
 v_\alpha(f) &:=  \sup \{\, 2^{\alpha k} v_k(f)\,:\, k\in\N  \}\;,\\
 \norm{f}_\alpha &:= \norm{f}_\infty + v_\alpha(f) \;,\\
\Hscr_\alpha(X) &:=\{\, f\in L^\infty(X)\,:\,
v_\alpha(f)<+\infty \, \}  \;. 
\end{align*}
The last set, $\Hscr_\alpha(X)$, is the space of H\"older continuous functions with exponent $\alpha$  w.r.t. the
distance $\dtil$ on $X$.
In fact it follows easily from  the definition that
 $$v_\alpha(f) = \sup_{x \neq x'} \frac{\abs{f(x)-f(x')}}{\dtil(x,x')^\alpha}\;. $$

\begin{proposition} \label{algebra:Ha}
For all $0\leq \alpha\leq 1$, 
 $(\Hscr_\alpha(X),\norm{\cdot}_\alpha)$
is a unital Banach algebra, and also a lattice. 
\end{proposition}

\begin{proof} To see that  $(\Hscr_\alpha(X),\norm{\cdot}_\alpha)$ is a normed algebra with unity, it is enough to verify the following inequalities:
\begin{align*}
  v_k(f g) &\leq \norm{f}_\infty  v_k(g) + \norm{g}_\infty  v_k(f) \;,\\
v_\alpha(f g) &\leq \norm{f}_\infty  v_\alpha(g) + \norm{g}_\infty  v_\alpha(f)\;.
\end{align*}
They imply that
$$ \norm{f g}_\alpha\leq \norm{f}_\alpha\,\norm{g}_\alpha \;,$$
and clearly  $\norm{\one}_\alpha=\norm{\one}_\infty + v_\alpha(\one)= 1+0=1$.
The proof that $(\Hscr_\alpha(X),\norm{\cdot}_\alpha)$ is a lattice and a Banach space is left as an exercise.
\end{proof}

\begin{definition}
\label{def Hscr alpha X minus}
We say that   $f:X\to\C$ is  {\em future independent}
if  $f(x)=f(y)$ for any  $x,y\in X$ such that  $x_k=y_k$ for all $k\leq 0$. Define the space
\begin{equation}\label{Hscr:alpha:Xminus}
 \Hscr_\alpha(X^-):=\left\{
\, f\in \Hscr_\alpha(X)\,\colon 
f \, \text{ is future independent } 
\, \right\}\;. 
\end{equation}
\end{definition}

The space $\Hscr_\alpha(X^-)$ is a closed sub-algebra of $\Hscr_\alpha(X)$, and hence a unital Banach algebra itself. 

Denote by $\Fcal^+$
the sub $\sigma$-field of $\Fcal$ generated by cylinders in non-negative coordinates. Likewise, denote by  $\Fcal^-$
the sub $\sigma$-field of $\Fcal$ generated by cylinders in non-positive coordinates. With this terminology, the subspace $\Hscr_\alpha(X^-)$ consists of all $\Fcal^-$-measurable functions in $\Hscr_\alpha(X)$.

The base LDT theorem below makes use of the standard notation
$\EE_\mu(\xi)=\int_X\xi\,d\mu$. 
This theorem is proved in section~\ref{random_base_ldt}.

\begin{theorem}\label{Base:LDT}
Let $(K,\mu)$ be a strongly mixing Markov system.
For any  $0<\alpha\leq 1$ and   $\xi\in\Hscr_\alpha(X^ -)$ there exist
 $C=C(\xi)>0$, $k=k(\xi)>0$ and $\varepsilon_0=\varepsilon_0(\xi)>0$
such that for all $0<\varepsilon<\varepsilon_0$,  $x\in\Sigma$ and  $n\in\N$,  
$$ \Pp_\mu\left[ \, \abs{\frac{1}{n}\,\sum_{j=0}^{n-1}\xi\circ T^j- \EE_\mu(\xi)} > \varepsilon \,\right]\leq C\,e^{ - {k\,\varepsilon^2 }\, n  } \;. $$
Moreover, the constants $C$, $k$ and $\varepsilon_0$ depend only on $K$ and $\norm{\xi}_\alpha$,
and hence can be kept constant when $K$ is fixed and $\xi$ ranges over any bounded set in $\Hscr_\alpha(X^ -)$.
\end{theorem}

The fiber LDT theorem,
proved in section~\ref{random_fiber_ldt},
has the following statement.

\begin{theorem}\label{Fiber:LDT}
Given a Markov system $(K,\mu)$ and   $A\in \randcocycles{\infty}{m}$, assume
\begin{enumerate}
\item[(1)]  $(K,\mu)$ is  strongly mixing,
\item[(2)]  $A$ is irreducible,
\item[(3)]  $L_1(A)>L_2(A)$.
\end{enumerate} 
Then  there exists   $\V$ neighborhood of $A$ in $\randcocycles{\infty}{m}$ and there exist
 $C>0$, $k>0$ and $\varepsilon_0>0$
such that for all\, $0<\varepsilon<\varepsilon_0$,  $B\in\V$ and  $n\in\N$,  
$$ \Pp_\mu\left[ \, \abs{\frac{1}{n}\,\log \norm{B^{(n)}} - L_1(B) } > \varepsilon \,\right]\leq C\,e^{ - {k\,\varepsilon^2 }\, n  } \;. $$
\end{theorem}

\subsection{The spectral method}
\label{random_sm}

Consider a Markov system $(K,\mu)$ on a compact metric space $\Sigma$.
Given some $\FF$-measurable measurable observable $\xi:\Sigma\to \R$,
let $\hat \xi:X^+\to \R$ be the $\Fcal^+$-measurable function $\hat \xi(x)=\xi(x_0)$.

Given $x\in\Sigma$, let $\Pp_x^+$ denote the probability on the measurable space $(X^+,\Fcal^+)$  that makes $\{e_n:X^+\to \Sigma\}_{n\geq 0}$ a Markov process with transition kernel $K$ and initial distribution with  point mass $\delta_x$ (see Definition~\ref{def measure Pp theta}).
Then 
$\{\hat \xi\circ T^n\}_{n\geq 0}$ is also a Markov process on $(X^+,\Fcal^+,\Pp_x^+)$.

\begin{definition}
We call {\em sum process} to the following sequence of random variables $\{S_n(\xi)\}_{n\geq 0}$
on $(X^+,\Fcal^+)$,
$$ S_n(\xi)(x):= \sum_{j=0}^{n-1} \hat \xi\circ T^j(x) = \sum_{j=0}^{n-1} \xi(x_j)  \;.$$
\end{definition}

\begin{definition}
An {\em observed Markov system on $(\Sigma,\FF)$}
is a triple $(K,\mu,\xi)$ where $(K,\mu)$ is a Markov system on $(\Sigma,\FF)$, and $\xi:\Sigma \to\R$ is an $\FF$-measurable function.
\end{definition}

\begin{definition}
We say that $\xi$ satisfies {\em LDT estimates of exponential type} if there exist positive constants
$C$, $k$ and $\varepsilon_0$ such that for all\, $n\in\N$, $0<\varepsilon<\varepsilon_0$ and $x\in\Sigma$,
$$  \Pp_x^+ \left\{ \, y\in X^+\,\colon \, \abs{\frac{1}{n}\,S_n(\xi)(y) -\EE_\mu(\xi)  } >\varepsilon \, \right\} \leq C\, e^{-n\,k\,\varepsilon^2} \;.$$

Given a class $\Xfrk$ of observed Markov systems
$(K,\mu,\xi)$  on a given measurable space $(\Sigma,\FF)$,
we say that $\Xfrk$ satisfies   {\em uniform} LDT estimates of exponential type if there exist positive constants
$C$, $k$ and $\varepsilon_0$ such that for every observed Markov system
$(K,\mu,\xi)\in  \Xfrk$, the observable $\xi$ satisfies LDT estimates of exponential type with constants $C$, $k$ and $\varepsilon_0$.
\end{definition}

\begin{definition}
Let  $\eta:X^+\to\R$ be a random variable
on  $(X^+,\Fcal^+)$.
The function $c(\eta,x,\cdot): \R\to \R$,
$$ c(\eta,x,t):=  \log \EE_x[ e^{t\, \eta}]$$
is called the {\em second characteristic function} of $\eta$, also known as the {\em cumulant generating function} of $\eta$  (see~\cite{Lukacs}).
\end{definition}

\begin{proposition}
\label{prop c(eta,x,t)}
Let $\eta:X^+\to\R$ be a $\Fcal^+$-measurable random variable.
Assume there exist $a,M>0$ such that for all $x\in\Sigma$,
$\max\{ \EE_x[ e^{a\,\eta}], \EE_x[  e^{a\,\eta}\,\vert \eta\vert ] \}\leq M$.
Then the cumulant generating function $c(\eta,x,\cdot)$ satisfies
\begin{enumerate}
\item[(1)]  $c(\eta,x,t)$ is well-defined and analytic for $t\in (-a,a)$,
\item[(2)]  $c(\eta,x,0)= 0$,
\item[(3)]  $\frac{d c}{dt}(\eta,x,0)= \EE_x(\eta)$,
\item[(4)]  $c(\eta,x,t)\geq t\,\EE_x(\eta)$, for all $t\in (-a,a)$,                                                      
\item[(5)]  the function $c(\eta,x,\cdot):(-a,a)\to\R$,  $t\mapsto c(\eta,x,t)$, \, is convex.
\end{enumerate}
\end{proposition}

\begin{proof}
For (1) notice that the assumptions imply that the parametric integral
$\EE_x(e^{z\,\eta})$ and its formal derivative
$\EE_x(e^{z\,\eta}\,\eta)$ are well-defined continuous functions on the disk $\vert z\vert <a$.
Since $c(\eta,x,0)=\log \EE_x(\one)=\log 1=0$, (2) follows. Property (3) holds because 
$\frac{d c}{dt}(\eta,x,0)= \EE_x(\eta\,\one)/\EE_x(\one)=\EE_x(\eta)$. The convexity (5) follows by  H\"older inequality, with conjugate exponents 
$p=1/s$ and $q=1/(1-s)$,  where $0<s<1$.
In fact, for all $t_1,t_2\in \R$,
\begin{align*}
c( \eta,x, s\,t_1 + (1-s)\,t_2)&=
\log \EE_x[ \left(e^{t_1\,\eta}\right)^s \,
\left(e^{t_2\, \eta}\right)^{1-s} ] \\
&\leq \log \left(\EE_x[e^{t_1\,\eta}] \right)^s 
\left(\EE_x[e^{t_2\, \eta}] \right)^{1-s}\\
&= s\,c( \eta,x,t_1) + (1-s)\,c( \eta,x,t_2) \;. 
\end{align*}
Finally, (2), (3) and (5) imply (4).
\end{proof}

Given an observable $\xi:\Sigma\to\R$, the function $c_n(\xi,x,\cdot):\R\to\R$ defined by
$$ c_n(\xi,x,t):= \log \EE_x[ e^{t\,S_n(\xi)}] \;,$$
is the cumulant generating function of $S_n(\xi)$.
Under general conditions, e.g., if $\xi$ is bounded, this function is analytic in $\C$, or at least analytic in a neighbourhood of $0$.

Let us write $\mathbb{D}_a(0)=\{\, z\in\C\,\colon\,
\vert z\vert <a\,\}$.

\begin{definition}
\label{def ucgf}
We call {\em limit cumulant generating function} of the process $\{S_n(\xi)\}_{n\geq 0}$ to any function
$c(\xi,\cdot):\mathbb{D}_a(0)\to\C$ such that 
there exist a constant $C>0$ and a numeric sequence $\{\delta_n\}_{n\geq 0}$ for which the following properties hold:
\begin{enumerate}
\item[(1)] $c_n(\xi,\cdot)$ is well defined and analytic on $\mathbb{D}_a(0)$,  for all $n\in\N$,
\item[(2)] $\abs{ n\,c(\xi,z)-c_n(\xi,x,z)} \leq C\,\abs{z}+\delta_n$,
for all $n\in\N$, $z\in\mathbb{D}_a(0)$ and $x\in\Sigma$,
\item[(3)] $\lim_{n\to +\infty} \delta_n=0$.
\end{enumerate}
\end{definition}

Before discussing why they exist, let us draw some conclusions from the existence of limit cumulant generating functions.

\begin{proposition}
Given an $\FF$-measurable observable $\xi:\Sigma\to\R$,
let  $c(\xi, z)$ be a limit cumulant generating function of the process $\{S_n(\xi)\}_{n\geq 0}$ on $\mathbb{D}_a(0)$. Then
\begin{enumerate}
\item[(1)] $z\mapsto c(\xi,z)$ is analytic on $\mathbb{D}_a(0)$,
\item[(2)]  $c(\xi, 0)= 0$,
\item[(3)]  $\frac{d c}{dt}(\xi,0)= \EE_\mu(\hat \xi)$,
\item[(4)]  $c(\xi,t)\geq t\,\EE_\mu(\hat\xi)$, for all $t\in \R$,
\item[(5)]  the function $c(\xi,\cdot):(-a,a) \to\R$,  $t\mapsto c(\xi,t)$, \, is convex.
\end{enumerate}
\end{proposition}

\begin{proof} 
The function $c(\xi,z)$ is analytic on $\mathbb{D}_a(0)$ because it is the uniform limit of the sequence of analytic functions $\frac{1}{n}\,c_n(\xi,x,z)$. This proves (1).

Item (2) follows directly from proposition~\ref{prop c(eta,x,t)} (2).

Consider now the sequence of analytic functions
$$  \hat c_n(\xi,z):= \int_\Sigma c_n(\xi,x,z)\,d\mu(x)\;. $$
Then
$$  \frac{d \hat c_n}{dt}(\xi,0)
= \int_\Sigma \EE_x[ S_n(\xi)]\,d\mu(x) = \EE_\mu(\hat\xi)\;. $$
Taking the limit identity (3) holds.

Since convexity is a closed  property, (5) follows from proposition~\ref{prop c(eta,x,t)} (5).

Finally, (2),(3) and (5) imply (4).
\end{proof}

Next proposition relates the existence of a limit cumulant generating function for the process $\{S_n(\xi)\}_{n\geq 0}$ with LDT estimates of exponential type for  $\xi$.

\begin{proposition}
\label{prop lcgd => LDT estimates}
Let $\xi:\Sigma\to\R$ be $\FF$-measurable observable,
and $c(\xi,z)$ be a limit cumulant generating function of the process $\{S_n(\xi)\}_{n\geq 0}$ on $\mathbb{D}_a(0)$.

Given $h>\frac{d^2 c}{dt^2}(\xi,0)$, there exist  $C,\varepsilon_0>0$ such that for all $n\in\N$, $x\in\Sigma$ and $0<\varepsilon<\varepsilon_0$,

$$ \Pp_x^+\left[ \; \abs{ \frac{1}{n}\, S_n(\xi) - \EE_\mu(\hat \xi)} > \varepsilon  \, \right] \leq C\, e^{-n\, \frac{\varepsilon^2}{2 h}} \;. $$
In other words, $\xi$ satisfies   LDT estimates of exponential type.
\end{proposition}

\begin{proof}
Let us abbreviate  $c(t)=c(\xi,t)$.
We can assume that  $c'(0)=\EE_\mu(\hat \xi)=0$. Otherwise we would work with
$\xi'=\xi-\EE_\mu(\xi)\,\one$, for which $\EE_{\mu}(\hat \xi')=0$. Notice that  the normalized process $\{S_n(\xi')\}_{n\geq 0}$ admits the
limit cumulant generating function
\,
$c(\xi', t)= c(t)-t\,\EE_\mu(\xi)=c(t)-t\,c'(0) $.

Since $h>c''(0)$, we can choose $0<t_0<a$ such that
for all $t\in (-t_0,t_0)$,
$$ 0\leq c(t)  < \frac{h\,t^2}{2}\;. $$

By definition~\ref{def ucgf}, for all $t\in (-t_0,t_0)$,
$$ \EE_x[ e^{t\,S_n(\xi)} ]=e^{c_n(\xi,x,t)} 
\leq e^{n\, c(t) + C\,\vert t\vert +\delta_n} \leq \frac{C_0}{2} \, e^{n\, \frac{h\,t^2}{2}} \;, $$
where 
$C_0:=2\,e^{C\, t_0 +\sup_{n\geq 0} \delta_n}$.
Thus, by Chebyshev's inequality we have forall  $\abs{t}<t_0$
$$
\Pp_x[\, S_n(\xi) >n\varepsilon\,] \leq e^{-t n \varepsilon} \EE_x[e^{t S_n(\xi)}] \leq \frac{C_0}{2}\, e^{- n\, \left( t  \varepsilon - \frac{h\,t^2}{2} \right) }\;. $$
Given $0<\varepsilon<\varepsilon_0:=h\,t_0$, pick  $t=\frac{\varepsilon}{\ctwo}\in ]0,t_0[$. 
This choice of $t$ minimizes the function $g(t)=e^{-\left( t  \varepsilon -\frac{\ctwo \,t^2}{2}\right)}$.
For this value of $t$ we obtain
$$ \Pp_x[\, S_n(\xi) >n\varepsilon\,]  \leq  \frac{C_0}{2}\, e^{- \frac{\varepsilon^2}{2\,\ctwo} \,n} \;.
$$
We can derive the same conclusion for  $-\xi$, 
because $c(\xi,-t)$ is a limit cumulant generating function of the process $\{S_n(-\xi)\}_{n\geq 0}$,
\begin{align*}
 \Pp_x[\, S_n(\xi) < - n\varepsilon\,]  &=  \Pp_x[\, S_n(-\xi) >n\varepsilon\,] \leq \frac{1}{2}\,C_0\,
 e^{- \frac{\varepsilon^2}{2\,\ctwo} \,n  }\;. 
\end{align*}
Thus, for all $x\in\Sigma$, $0<\varepsilon<\varepsilon_0$ and $n\in\N$,
\begin{align*}
\Pp_x[\; \vert S_n(\xi)\vert  >n\varepsilon\,] \leq   C_0\,e^{- \frac{\varepsilon^2}{2\,\ctwo}\,n } \;. 
\end{align*}
\end{proof}

\begin{remark}
To obtain a sharp upper bound on the rate function
for the large deviations of the process $S_n(\xi)$
we should have used the Legendre transform of the convex function
$c(t)-t\,c'(0)$. Here because we do not care about sharp estimates, but mainly to avoid dealing with the degenerate case where $c(t)$ is not strictly convex, we have replaced 
$c(t)-t\,c'(0)$ by its upper bound $\frac{h\,t^2}{2}$
on the small neighborhood $(-t_0,t_0)$, which is always strictly convex.
\end{remark}

Consider now a topological space $\Xfrk$ of observed Markov systems $(K,\mu,\xi)$, on a given measurable space $(\Sigma,\FF)$.

Denote by $\Hscr(\mathbb{D}_a(0))$ the Banach space of
analytic functions $f:\mathbb{D}_a(0)\to\C$ with a continuous extension up to the disk's closure. Endow this space with the usual max norm
$\norm{f}_\infty=\max_{\vert z\vert\leq a} \vert f(z) \vert$.

\begin{corollary}
\label{coro uniform LDT estimates}
Assume there is continuous map
$c:\Xfrk\to \Hscr(\mathbb{D}_a(0))$ such that
\begin{enumerate}
\item[(a)] for each $(K,\mu,\xi)\in\Xfrk$, the function
$c(\xi,z):=c(K,\mu,\xi)(z)$ is a limit cumulant generating function of the process $\{S_n(\xi)\}_{n\geq 0}$ on $\mathbb{D}_a(0)$,
\item[(b)] the parameters
$C$ and $\delta_n$ in definition~\ref{def ucgf} can be chosen uniformly in $\Xfrk$.
\end{enumerate}

Then 
\begin{enumerate}
\item[(1)] For each $(K,\mu,\xi)\in\Xfrk$ there exists  a neighborhood $\V$ in $\Xfrk$  such that $\V$ satisfies uniform LDT estimates of exponential type.

\item[(2)] If there exists $h>0$ such that
$\frac{d^2}{dt^2}c(\xi,0)<h$ for all $(K,\mu,\xi)\in\Xfrk$ then $\Xfrk$  satisfies uniform LDT estimates of exponential type.
\end{enumerate}
\end{corollary}

\begin{proof}
Given $(K_0,\mu_0,\xi_0)\in\Xfrk$,
let $c_0(t):=c(K_0,\mu_0,\xi_0)(t)$, and take
$h>c_0''(0)$. By continuity of $c:\Xfrk\to \Hscr(\mathbb{D}_a(0))$ there exist a neighborhood $\V$ of $(K_0,\mu_0,\xi_0)$ in $\Xfrk$ and $t_0>0$ such that for any $(K,\mu,\xi)\in \V$,
the function $c(\xi,z):=c(K,\mu,\xi)(z)$ satisfies for all $t\in (-t_0,t_0)$,
$$  c(\xi,t)-t\, \frac{d c}{dt} (\xi,0) < \frac{h\,t^2}{2}\;. $$
The argument used to prove
proposition~\ref{prop lcgd => LDT estimates}
shows that $\V$ satisfies uniform LDT estimates of exponential type.
\end{proof}

The strategy to meet the assumptions of corollary~\ref{coro uniform LDT estimates},
i.e., to prove the existence of a limit cumulant generating function for the process $\{S_n(\xi)\}_{n\geq 0}$,  is a spectral method that we describe now.

 Define a family of Laplace-Markov  operators
$$ (Q_t f)(x)=(Q_{K,\xi,t} f)(x):= \int_\Sigma f(y)\,e^{t\, \xi(y)}\,K(x,dy) \;, $$
on some appropriate Banach space $\B$, embedded in $L^\infty(\Sigma,\FF)$, and containing the constant functions.
Notice that by definition $(Q_t \one)(x)=\EE_x[e^{t\,\hat \xi}]$. Hence, iterating this relation
we obtain the following formula for the moment generating function of $S_n(\xi)$: for all $x\in\Sigma$ and $n\in\N$,
$$  \EE_x[ e^{t\,S_n(\xi)} ] = (Q_t^n \one)(x) \;. $$
For $t=0$, the operator $Q_0: \B\to \B$,
is a Markov operator. In particular it is a positive operator which fixes the constant functions, e.g., $Q_0\one=\one$, and  whose spectrum is contained in the closed unit disk.
The key ingredient to estimate the moment generating function $\EE_x[e^{t\,S_n(\xi)}]$ via this spectral approach is the assumption that
the operator $Q_0: \B\to \B$ is {\em quasi-compact} and {\em simple}. This means that the eigenvalue $1$ of $Q_0$ is simple and there exists a spectral gap separating this eigenvalue from the rest of spectrum inside the open unit disk.
Under this hypothesis, $Q_t$ is a positive operator,  whenever defined, and there exists a unique eigenfunction $v(t)\in\B$ such that $Q_t v(t)=\lambda(t)\, v(t)$, normalized by $\EE_\mu[v(t)]=1$, and corresponding to a positive eigenvalue $\lambda(t)$ of $Q_t$.
Hence, because the functions $t\mapsto \lambda(t)$
and $t\mapsto v(t)$ are continuous in $t$ (in fact analytic), we have
$$\EE_\mu[ e^{t\,S_n(\xi)} ] =
\int (Q_t^n \one)\,d\mu \approx \int  Q_t^n v(t) \,d\mu  =\int \lambda(t)^n v(t)\,d\mu =   \lambda(t)^n\;. $$
From this relation we infer that
$c(t)=\log \lambda(t)$ is a limit cumulant generating function for the process $S_n(\xi)$.
Therefore, by proposition~\ref{prop lcgd => LDT estimates}, $\xi$ satisfies LDT estimates of exponential type.

To obtain uniform LDT estimates, through corollary~\ref{coro uniform LDT estimates}, we assume some weak  continuous dependence of the family of operators $t\mapsto Q_{K,\xi,t}$ on the observed Markov system  $(K,\mu,\xi)$, which implies that the eigenvalue function $\lambda(t)\in \Hscr(\mathbb{D}_a(0))$ also depends continuously on    $(K,\mu,\xi)$.

\subsection{Literature review}
\label{random_lit_review}
%Short literature review:
We mention   briefly some of the origins of this subject.

One is the aforementioned Furstenberg's work,
started with the proof  by H. Furstenberg and H. Kesten 
 of a  law of large numbers for random i.i.d. products of matrices ~\cite{Furstenberg-Kesten}, and later abstracted by Furstenberg to a seminal theory on  random products in semisimple Lie groups ~\cite{Furstenberg-63}. In this context, a first central limit theorem  was proved by V. N. Tutubalin in ~\cite{Tutubalin}.
Since its origin, the scope of Furstenberg's theory has been greatly  extended
by many contributions. See for instance the
book of  A. Raugi~\cite{Raugi-book} and  
Y. Guivarc'h and A. Raugi's paper  ~\cite{GR}.

Another source is a central limit theorem of  S.V.Nagaev for
stationary Markov chains (see~\cite{Nagaev}). In his approach Nagaev uses the spectral properties of a  quasi-compact Markov operator acting on some space of bounded measurable functions.
%See also the work of T. Kaijser~\cite{Kaijser}.
This method was used by E. Le Page to obtain  more general central limit theorems, as well as a large deviation principle, for random i.i.d. products of matrices ~\cite{LP}.
%In this context Le Page  also proved, using the same technique, the continuity of the first Lyapunov exponent, under certain irreducibility and contracting assumptions (see~\cite{LP2}).
Later P. Bougerol extended Le Page's approach, proving similar
results for Markov type random  
products of matrices (see ~\cite{Bou}).

The book of P. Bougerol and J. Lacroix ~\cite{BL},
on  random i.i.d. products of matrices, is an excellent introduction on the subject  in  ~\cite{LP,Bou}.
More recentely, the book of  H. Hennion and L. Herv\'e  ~\cite{HH} describes a powerful abstract setting where the
method of Nagaev can be applied to derive limit theorems.
It contains several applications, including to dynamical systems and linear cocycles, that illustrate the method.
%In subsection ~\ref{abstract:setting} we specialize the setting in~\cite{HH}
%to prove an abstract LDT theorem which is still enough for our purposes.

\bigskip

\section{An abstract setting}
\label{random_as}

In this section we specialize an abstract setting in~\cite{HH}, from which we derive an abstract theorem on the existence of uniform LDT estimates for Markov processes.

\subsection{The assumptions}
\label{random_setting}
\newcommand{\Bscr}{\mathscr{B}}

Let $\B$ be a  Banach space, and $\mathcal{L}(\B)$ denote 
 the Banach algebra of bounded 
linear operators $T:\B\to \B$.
Given $T\in \mathcal{L}(\B)$,
we denote its spectrum by $\sigma(T)$,
and its spectral radius by 
$$\rho(T)=\lim_{n\to+\infty} \norm{T^ n}^{1/n}
= \inf_{n\geq 0} \norm{T^ n}^{1/n}\;.$$

\begin{definition}
The operator $T$ is called  quasi-compact  if there is a $T$-invariant decomposition
$\B=F\oplus\Hscr$ such that $\dim F<+\infty$ and  
the spectral radius of $T\vert_\Hscr$ is (strictly) less than the absolute
value $\vert \lambda\vert $ of any eigenvalue $\lambda$ of $T\vert_F$.
$T$ is called  quasi-compact and simple   when  furthermore  $\dim F=1$.
In this case $\sigma(T\vert_F)$ consists of a single simple eigenvalue referred to as
the  maximal eigenvalue of $T$.
\end{definition}

Consider a  Markov system $(K,\mu)$ on a compact metric space $\Sigma$.

\begin{definition}
The following linear operator  is called a  {\em Markov operator}
$$ (Q f)(x)= (Q_K f)(x):= \int_X f(y)\, K(x,dy)\;.$$
\end{definition}
It operates on $\FF$-measurable functions on $\Sigma$,
 mapping $L^p$ functions to $L^p$ functions, for any $1\leq p\leq \infty$.
We shall write $Q$ instead of $Q_K$ when  the kernel $K$ is fixed.

\begin{definition}
The following  linear operator is called a  {\em Laplace-Markov operator}
$$ (Q_{\xi} f)(x) = (Q_{K,\xi} f)(x):= \int_X f(y)\, e^{\xi(y)}\, K(x,dy)\;.$$
\end{definition}
It also operates on $\FF$-measurable functions on $\Sigma$,
 but the domain of $Q_\xi$ depends also on the observable $\xi$.

\begin{proposition} Given a  Markov system $(K,\mu)$
the following are equivalent:
\begin{enumerate}
\item[(a)] $(K,\mu)$ is strongly mixing,
\item[(b)] $Q_K:L^\infty(\Sigma,\FF)\to L^\infty(\Sigma,\FF)$ is quasi-compact and simple.
\end{enumerate}
\end{proposition}

\begin{proof}
If $(K,\mu)$ is strongly mixing, by definition~\ref{Strong:mixing} 
there exist constants $C>0$ and $0<\rho<1$ such that
for all $f\in L^\infty(\Sigma)$,
$$ \norm{ (Q_K)^n f -\langle f,\mu\rangle\, \one }_\infty \leq C\,\rho^n \norm{f}_\infty\;. $$
Defining
$$ H_0=\{\, f\in L^\infty(\Sigma)\,:\,
\langle f,\mu\rangle=0\; \} \;,$$
since $(Q_K)^\ast\mu=\mu$, this subspace is $Q_K$-invariant.
Thus, we have a 
$Q_K$-invariant decomposition 
$ L^\infty(\Sigma)=\langle \one\rangle \oplus H_0$ 
such that  $ \norm{ (Q_K)^n\vert_{H_0} }\leq C\,\rho^n$.
This implies that $r( Q_K \vert_{H_0} )\leq \rho <1$.

Conversely, if $Q_K:L^\infty(\Sigma)\to L^\infty(\Sigma)$ is quasi-compact and simple, there exists a $Q_K$-invariant decomposition
$ L^\infty(\Sigma)=\langle \one\rangle \oplus H_0$ such that $r( Q_K \vert_{H_0} )  <1$.
By the Hahn-Banach Theorem there is a bounded linear functional $\Lambda:L^\infty(\Sigma)\to\R$ such that
 $\Lambda(\one)=1$, and $\Lambda(f)=0$ for all $f\in H_0$.
We claim that $\Lambda$ is positive functional, i.e., $\Lambda(f)\geq 0$ whenever $f\geq 0$.
Take any function $f\in L^\infty(\Sigma)$ such that $f\geq 0$, and write $f=c\,\one + h$
with $h\in H_0$. Since $Q_K$ is a positive operator we have
$$ c\,\one =\lim_{n\to+\infty} ( c\,\one + (Q_K)^n h ) =  \lim_{n\to+\infty}  (Q_K)^n f \geq 0 \;,$$
which implies that $c=\Lambda(f)\geq 0$. 
Hence  $\Lambda$ is positive.
By the Riez-Markov-Kakutani Theorem there is a probability measure $\mu$ on $\Sigma$ such that
$\Lambda(f)=\int_\Sigma f\,d\mu$, for all $f\in L^\infty(\Sigma)$. 

Let us prove that  $\mu$ is $K$-stationary.
Given $f\in L^\infty(\Sigma)$,  write $f=c\,\one + h$, with $h\in H_0$.
Hence $Q_K f = c\,\one + Q_K h$ with $Q_K h\in H_0$.
This proves that $\mu$ is stationary,
$$\int_\Sigma (Q_K f)\,d\mu = \Lambda(Q_K f) = c = \Lambda(f)=\int_\Sigma f\,d\mu\;. $$

Now, because $H_0$ is the kernel of $\Lambda :L^\infty(\Sigma)\to\R$, 
we get that for all $f\in L^\infty(\Sigma)$, 
$f\in H_0$\, $\Leftrightarrow$\, $\langle f,\mu\rangle=0$.
Thus
$f- \langle f,\mu\rangle\,\one\in H_0$,  
and taking $r( Q_K \vert_{H_0} )< \rho <1$, there is a constant  $C>0$ such that
\begin{align*}
\norm{ (Q_K)^n f -\langle f,\mu\rangle\, \one }_\infty &=
\norm{ (Q_K)^n [ f -\langle f,\mu\rangle\, \one] }_\infty \\
&\leq C\,\rho^n \,\norm{f -\langle f,\mu\rangle\, \one}_\infty\\
&\leq 2\,C \,\rho^n \,\norm{f}_\infty\;.
\end{align*}
This proves that $(K,\mu)$ is strongly mixing.
\end{proof}

\medskip

We discuss now a setting, consisting of 
the assumptions (B1)-(B7) and (A1)-(A4) below,
where an abstract LDT theorem is proved, and from which 
theorems~\ref{Base:LDT} and ~\ref{Fiber:LDT} will be
deduced.
The context here specializes a  more general setting in~\cite{HH}.

Let $(\Xfrk,\dist)$ be  a metric space of observed Markov systems $(K,\mu,\xi)$ over the compact metric space $(\Sigma,d)$. 
Besides $\Xfrk$, this setting consists of a scale of complex Banach algebras $(\B_\alpha,\norm{\cdot}_\alpha)$ indexed in $\alpha\in [0,1]$,
 where each  $\B_\alpha$ is a space of bounded Borel measurable functions on $\Sigma$. We assume  that
there exist seminorms $v_\alpha:\B_\alpha\to [0,+\infty[$ such that for all  
 $0\leq \alpha  \leq 1$,
\begin{enumerate}
\item[(B1)] $ \norm{f}_\alpha = v_\alpha(f) + \norm{f}_\infty$, for all $f\in\B_\alpha$,
\item[(B2)] $\B_{0} = L^\infty(\Sigma)$, and $\norm{\cdot}_0$ is equivalent to $\norm{\cdot}_\infty$,
\item[(B3)] $\B_{\alpha}$ is a lattice, i.e., if $f\in\B_{\alpha}$ then $\overline{f}, \abs{f}\in\B_{\alpha}$,
\item[(B4)] $\B_{\alpha}$ is a  Banach algebra with unity $\one\in\B_{\alpha}$ and $v_\alpha(\one)=0$.
\end{enumerate}
Assume also that this family  is a {\em scale of normed spaces} in the sense that
 for all $0\leq \alpha_0<\alpha_1<\alpha_2 \leq 1$ (see~\cite{KP})
\begin{enumerate}
\item[(B5)] $\displaystyle \B_{\alpha_2}\subset \B_{\alpha_1}\subset \B_{\alpha_0}$,
\item[(B6)] $\displaystyle v_{\alpha_0}(f)\leq v_{\alpha_1}(f)\leq v_{\alpha_2}(f)$, for all $f\in \B_{\alpha_2}$,
\item[(B7)] $\displaystyle v_{\alpha_1}(f)\leq 
v_{\alpha_0}(f)^{\frac{\alpha_2-\alpha_1}{\alpha_2-\alpha_0}}
v_{\alpha_2}(f)^{\frac{\alpha_1-\alpha_0}{\alpha_2-\alpha_0}}$, for all $f\in \B_{\alpha_2}$.
\end{enumerate}

An example of a scale of Banach algebras satisfying (B1)-(B7)  are the spaces 
of $\alpha$-H\"older continuous functions on $(\Sigma,d)$. The norms on these spaces are defined as follows:  for all $\alpha\in ]0,1]$ 
and $f\in L^ \infty(\Sigma)$, let
$$ \norm{f}_\alpha := v_\alpha(f) + \norm{f}_\infty, \; \text{ with }\;
v_\alpha(f):= \sup_{\substack{x,y\in\Sigma\\x\neq y}} \frac{\abs{f(x)-f(y)}}{d(x,y)^ \alpha}\;.$$

\begin{proposition}\label{Holder:B1:B7}
If   $(\Sigma,d)$ has diameter $\leq 1$ then the family of spaces  
$$\Hscr_\alpha(\Sigma):=\{\, f\in L^ \infty(\Sigma)\,:\,
v_\alpha(f)<+\infty\,\},\;  \alpha\in [0,1] $$
satisfies (B1)-(B7). 
\end{proposition}
\begin{proof}
(B1) holds by definition of the H\"older norm $\norm{\cdot}_\alpha$.
For (B2) notice that $v_0(f)$ measures the oscillation of $f$, and hence 
$v_0(f)\leq 2\,\norm{f}_\infty$.
Property (B3) is obvious.
Assumption (B4) follows from the following inequality
$$v_\alpha(f\,g)\leq \norm{f}_\infty\,v_\alpha(g) + \norm{g}_\infty\,v_\alpha(f)\;,$$
that holds for all $f,g\in L^\infty(\Sigma)$.
The monotonicity properties (B5) and (B6) are straightforward to check.
Finally, assumption (B7) follows from the convexity of the function  $\alpha\mapsto \log v_\alpha(f)$. Given $\alpha_1,\alpha_2, s\in [0,1]$,
\begin{align*}
\log v_{s\,\alpha_1+(1-s)\,\alpha_2}(f) &= \log \sup_{x\neq y}
\frac{\vert f(x)-f(y)\vert ^{s+(1-s)}}{d(x,y)^{s\,\alpha_1+(1-s)\,\alpha_2}} \\
&\leq  \log \left( \sup_{x\neq y}
\frac{\vert f(x)-f(y)\vert }{d(x,y)^{\alpha_1}}\right)^s\,
\left( \sup_{x\neq y}
\frac{\vert f(x)-f(y)\vert }{d(x,y)^{\alpha_2}}\right)^{1-s}\\
&= s\,\log v_{\alpha_1}(f) + (1- s)\,\log v_{\alpha_2}(f)\;.
\end{align*}
\end{proof}

\medskip

We make now a second set of assumptions that rule the action of the Markov operators, associated to observable Markov systems $(K,\mu,\xi)\in \Xfrk$, on the Banach algebras $\B_\alpha$.

Assume there exists an interval 
$[\alpha_1,\alpha_0]\subset (0,1]$ 
with $\alpha_1<\frac{\alpha_0}{2}$ such that for all $\alpha\in [\alpha_1,\alpha_0]$ the following properties hold:

\begin{enumerate}
\item[(A1)] $(K,\mu,-\xi)  
\in \Xfrk$ whenever
$(K,\mu,\xi)\in \Xfrk$.

\item[(A2)]  The Markov operators $Q_K:\B_\alpha\to\B_\alpha$ are uniformly quasi-compact and simple. More precisely, there exist constants $C>0$ and
$0<\sigma<1$ such that for all $(K,\mu,\xi)\in\Xfrk$ and $f\in \B_\alpha$,
$$ \norm{Q_{K}^n f -\langle f,\mu \rangle \one }_\alpha\leq C\,\sigma^n\, \norm{f}_\alpha\;. $$

\item[(A3)]  The operators $Q_{K,z\,\xi}$ act continuously
on the Banach algebras $\B_\alpha$, uniformly in $(K,\mu,\xi)\in\Xfrk$ and $z$ small. More precisely, we assume there are constants $b>0$ and $M>0$ such that
for $i=0,1,2$,  $\vert z \vert <b$ and $f\in\B_\alpha$,
$$Q_{K,z\,\xi}(f\,\xi^i)\in\B_\alpha\quad \text{ and }\quad 
 \norm{ Q_{K,z\,\xi}(f\, \xi^i) }_\alpha \leq M\,\norm{f}_\alpha\;.
$$

\item[(A4)]  The family of functions
$\Xfrk\ni (K,\mu,\xi)\mapsto Q_{K,z\,\xi}$, indexed in  $\vert z \vert \leq b$,
is  H\"older equi-continuous in the sense that there exists $0<\theta\leq 1$ such that for all
$\vert z \vert \leq b$, $f\in\B_\alpha$ and $(K_1,\mu_1,\xi_1),(K_2,\mu_2,\xi_2)\in\Xfrk$,
$$ \norm{Q_{K_1, z \,\xi_1} f - Q_{K_2,z\,\xi_2 } f }_\infty \leq
		 M\, \norm{f}_\alpha\, \dist((K_1,\mu_1,\xi_1),(K_2,\mu_2,\xi_2))^ \theta \;.$$

\end{enumerate}

The interval $[\alpha_1,\alpha_0]$ will called as the {\em range} of the scale of Banach algebras. 
In the fiber LDT theorem we will need to take $\alpha_0$ small enough to have contraction in (A2), but at the same time we need $\alpha_1$ bounded away from $0$ to have uniformity in this contraction. 
The need for the condition $\alpha_1<\frac{\alpha_0}{2}$ is explained in remark~\ref{Holder continuity}.

The positive constants  $C$, $\sigma$, $M$,  $b$ and $\theta$   above will be called the {\em setting constants}.

Examples of contexts satisfying all  assumptions (B1)-(B7) and
(A1)-(A4) are provided by the applications in sections ~\ref{random_base_ldt} and ~\ref{random_fiber_ldt}.

%Coments on the assumptions:
The symmetry assumption (A1) allows us to reduce deviations below average to
 deviations above average, thus shortening the arguments.

(A2) is the main assumption: all Markov operators
$Q_K:\B_\alpha\to\B_\alpha$ are quasi-compact and simple, uniformly in $(K,\mu,\xi)\in\Xfrk$. This will imply that, possibly decreasing $b$, all Laplace-Markov operators $Q_{K,z\,\xi}:\B_\alpha\to\B_\alpha$ are also quasi-compact and simple, uniformly in $(K,\mu,\xi)\in\Xfrk$ and $\abs{z}<b$. 

(A3) is a regularity assumption.
The operators   $Q_{K, z\,\xi}$ act  continuously on $\B_\alpha$,
uniformly in $(K,\mu,\xi)\in\Xfrk$ and $\abs{z}<b$. 
Moreover, it implies that $\mathbb{D}_b\ni z\mapsto Q_{K, z\,\xi}\in \mathcal{L}(\B_\alpha)$,
is an  analytic function. 
 
Finally, (A4) implies that the function  $(K,\mu,\xi)\mapsto \lambda_{K,\xi}(z)$ is  uniformly H\"older continuous. 
Here $\lambda_{K,\xi}(z)$ denotes the maximal eigenvalue of
$Q_{K, z\,\xi}$,

These facts follow from the propositions stated and proved in the rest of this subsection.

Hypothesis (A3) implies that 
$Q_{K, z\xi}\in \mathcal{L}(\B_\alpha)$, for all $z\in\mathbb{D}_b$. 
In particular the function
$Q_{K,\ast\xi}:\mathbb{D}_b\to \mathcal{L}(\B_\alpha)$,
$z\mapsto Q_{K,z\xi}$,  is well-defined,
 for every $(K,\mu,\xi)\in \Xfrk$.

\begin{proposition} \label{Qz:analytic} 
The function $Q_{K,\ast\xi}:\mathbb{D}_b\to \mathcal{L}(\B_\alpha)$  is analytic with
$$    \frac{d}{dz} Q_{K, z\,\xi} (f)=Q_{K, z\,\xi}(f\,\xi) \; \text{ for }\,
f\in\B_\alpha\,,$$
for all $(K,\mu,\xi)\in\Xfrk$, and $\alpha_1\leq \alpha\leq\alpha_0$.
\end{proposition}

\begin{proof}
Given $b\in\R$, for all $z,z_0\in\C$,
$$ \frac{ e^{z\,b}-e^{z_0\,b} }{z-z_0} - b\, e^ {z_0\,b} 
= \int_{z_0}^{z} b^2\,e^ {\zeta\,b}\,\frac{z-\zeta}{z-z_0}\, d\zeta\;.$$
This is the first order Taylor remainder formula for
$h(z)=e^{b\,z}$ at $z=z_0$. To shorten notation we write $Q_z$ for
$Q_{K,z\,\xi}$. Replacing $b$ by $\xi(y)$,
multiplying by $f(y)\,K(x,dy)$ and integrating over $\Sigma$
we get
$$ \frac{ Q_z f - Q_{z_0} f}{z-z_0} - Q_{z_0} (f\,\xi) 
= \int_{z_0}^{z} Q_{\zeta}(f\,\xi^2)\,\frac{z-\zeta}{z-z_0}\, d\zeta\;.$$
Hence, by (A3), for all $z\in\mathbb{D}_b$,
\begin{align*}
\norm{ \frac{ Q_z f - Q_{z_0} f}{z-z_0} - Q_{z_0} (f\,\xi) }_\alpha
&\leq  \int_{z_0}^{z} \norm{ Q_{\zeta}(f\,\xi^2)}_\alpha\,
\frac{\abs{z-\zeta}}{\abs{z-z_0}} \, \vert d\zeta\vert\\
&\leq M\,\norm{f}_\alpha\,\abs{z-z_0}\;,
\end{align*}
which proves that  the following limit exists in $\mathcal{L}(\B_\alpha)$,
$$ \lim_{z\to z_0} \frac{ Q_z - Q_{z_0} }{z-z_0} = Q_{z_0} (\xi\,\cdot)\;.$$
Notice that (A3) also implies the operator
$Q_{z_0}(\xi\,\cdot)(f):=Q_{z_0}(\xi\,f)$ is in 
$\mathcal{L}(\B_\alpha)$.
\end{proof}

Next proposition focus  on  the quasi-compactness and simplicity of  $Q_z=Q_{K, z\,\xi}$,
and is proved using arguments in ~ \cite{LP,Bou}.

\begin{proposition} \label{abstr:qcs}
Consider a metric space $\Xfrk$  of observed Mar\-kov systems 
satisfying  (A1)-(A4) in the range $[\alpha_1,\alpha_0]\subset (0,1]$ with setting constants $C$, $\sigma$, $M$,  $b$ and $\theta$.

 Given $\varepsilon>0$ there exist  $C',M'>0$  and
$0<b_0<b$ such that the following statement holds: 
for all $(K,\mu,\xi)\in\Xfrk$, $z\in\mathbb{D}_{b_0}$ and
$\alpha_1\leq \alpha\leq \alpha_0$
there exist: a one dimensional subspace
$E_z=E_{K, z\,\xi}\subset \B_\alpha$, a hyperplane $H_z=H_{K, z\,\xi}\subset \B_\alpha$, a number $\lambda(z)=\lambda_{K,\xi}(z)\in\C$,
and a linear map $P_z=P_{K, z\,\xi}\in \mathcal{L}(\B_\alpha)$  
such that 
\begin{enumerate}
\item[(1)] $\B_\alpha=E_{z}\oplus H_{z}$ is a $Q_{z}$-invariant decomposition,
\item[(2)] $P_{z}$ is a projection onto $E_{z}$, parallel to $H_{z}$,
\item[(3)]  $Q_{z}\circ P_{z} = P_{z}\circ Q_{z} =\lambda(z)\,P_{z}$,
\item[(4)] $Q_{z} f = \lambda(z)\, f$ for all $f\in E_{z}$,
\item[(5)] $z\mapsto \lambda(z)$ is analytic in a neighborhood of $\overline{\mathbb{D}}_{b_0}$,
\item[(6)] $\displaystyle \abs{\lambda(z)}\geq 1-\varepsilon$.
\end{enumerate}
Furthermore, for all  $f\in\B_\alpha$, 
\begin{enumerate}
\item[(7)] $\displaystyle \norm{ Q_{z}^n f - \lambda(z)^n\,  P_{z} f }_\alpha
\leq C'\, (\sigma + \varepsilon)^n\,\norm{f}_\alpha$,
\item[(8)] $\displaystyle \norm{ P_z\,f }_\alpha \leq C'\,\norm{f}_\alpha$,
\item[(9)] $\displaystyle \norm{ P_z\,f -P_0\,f }_\alpha \leq C' \, \abs{z}\,\norm{f}_\alpha$,
\end{enumerate}
and for all $z\in\mathbb{D}_{b_0}$
and  $(K_1,\mu_1,\xi_1), (K_2,\mu_2,\xi_2)\in\Xfrk$,  
\begin{enumerate}
\item[(10)] $\displaystyle \abs{\lambda_{K_1,\xi_1}(z) - \lambda_{K_2,\xi_2}(z) } \leq M'\, d((K_1,\mu_1,\xi_1),(K_2,\mu_2,\xi_2))^ {\frac{\theta}{2}} $.
\end{enumerate}
\end{proposition}

\medskip

Given $(K,\mu,\xi)\in\Xfrk$, define the  operators
\begin{align}\label{Pz}
 P_{z} = P_{K,z\xi} := \frac{1}{2\pi i} \int_{\Gamma_1} R_{z}(w) \, dw  \\
 L_{z} = L_{K,z\xi} := \frac{1}{2\pi i} \int_{\Gamma_1} w\,R_{z}(w) \, dw  \label{Lz}\\
 N_{z} = N_{K,z\xi} := \frac{1}{2\pi i} \int_{\Gamma_0} w\,R_{z}(w) \, dw  \label{Nz}
\end{align}
where $\Gamma_0$ and $\Gamma_1$ are the positively oriented circles 
\begin{align*}
\Gamma_0 &=\{\, w\in\C\,:\, \abs{w}=\frac{1+2\sigma}{3}\,\}\, ,\\
\Gamma_1 &=\{\, w\in\C\,:\, \abs{w-1}=\frac{1-\sigma}{3}\,\}\,,
\end{align*}
and $R_{z}(w) = R_{K,z\xi}$ stands for the resolvent of  $Q_{K, z\xi}$,
$$R_{z}(w):= \left( w\,I- Q_{K, z\xi} \right)^{-1}\;.$$

\begin{lemma} \label{inverses} 
Given a normed space $(\B,\norm{\cdot})$ and linear operators $T,T_0\in \mathcal{L}(\B)$, \\
if \, $T_0$ \, is invertible with  \,
$\norm{T_0^{-1}}\leq C$\, and \,$\norm{T-T_0}\leq \varepsilon$ \,
then
\begin{enumerate}
\item $T$ is invertible, with
$\displaystyle \norm{T^{-1}}\leq \frac {C}{1-C\,\varepsilon}$,
\item $\displaystyle \norm{T^{-1} - T_0^{-1}}\leq \frac{ C^2 }{1-C\,\varepsilon}\, \norm{T-T_0}$.
\end{enumerate}
\end{lemma}

\begin{proof}
Since
$ T^ {-1} =\sum_{n=0}^ \infty (-1)^n\,(T_0^{-1}\, (T-T_0))^n \, T_0^{-1}$, we have
$$\norm{ T^ {-1} }\leq 
\sum_{n=0}^ \infty  \norm{T_0^{-1}}^{n+1} \norm{T-T_0}^n   =\frac{\norm{T_0^{-1}}  }{1-\norm{T_0^{-1}}  \norm{T-T_0} } 
\leq \frac{C }{1-C\,\varepsilon }  \;. $$
For (2) use the formula\,
$T^{-1}-T_0^{-1} = - T^{-1}\, (T-T_0)\, T_0^{-1}$.
\end{proof}

\begin{lemma} \label{spectral:bounds}
There exist  constants $C_0>0$ and $0<b_0<b$,  
depending only on $C,M,\sigma$ and $b$, such that for 
$(K,\mu,\xi)\in\Xfrk$,  $z\in\mathbb{D}_{b_0}$, and any of the five
operators $T_{z}= Q_{K, z\xi}$,  $L_{z}$, $N_{z}$, $P_{z}$, and $R_{z}(w)$  with $w\notin {\rm int}(\Gamma_0)\cup {\rm int}(\Gamma_1)$, 
\begin{enumerate}
\item $\displaystyle \norm{T_{z}} \leq C_0$,
\item $\displaystyle \norm{T_{z} - T_{0}}\leq C_0\,\abs{z}$.
\end{enumerate}
\end{lemma}

\begin{proof}
First note that $\norm{L_{0}}=\norm{P_{0}}=1$ and $\norm{N_{0}}\leq C\,\sigma$, so that $\norm{Q_{0}}=\norm{L_{0}+N_{0}}\leq 1 + C\,\sigma$.
Let us go through the given operators, one at a time. 
Assume $0<b_0<b$ is small and take $z\in\mathbb{D}_{b_0}$.
For $Q_{K, z\,\xi}$,
item (1) follows from assumption (A3), taking $C_0:=M$,
while (2) follows from (A3) and  Proposition~\ref{Qz:analytic} with the same constant.
For the operator $R_{z}(w)$, we have
\begin{align*}
R_{0}(w) &= w^{-1}\, (I- w^{-1}\,Q_{0})^{-1} =  w^{-1}\,\sum_{n=0}^\infty \frac{Q_{0}^n}{w^n}\\
&= w^{-1}\,\sum_{n=0}^\infty \frac{P_{0}}{w^n}  +   w^{-1}\,\sum_{n=0}^\infty \frac{N_{0}^n}{w^n} =   \frac{P_{0}}{w-1}    +   \sum_{n=0}^\infty \frac{N_{0}^n}{w^{n+1}}\;.
\end{align*}
Notice also that $w\notin {\rm int}(\Gamma_0)\cup {\rm int}(\Gamma_1)$
implies   $\abs{w-1}\geq \frac{1-\sigma}{3}$ and
$\abs{w}\geq \frac{1+2\sigma}{3}$, and hence
\begin{align*}
\norm{R_{0}(w)} &\leq  
\frac{\norm{P_{0}}}{\abs{w-1}}   +  \frac{C}{\abs{w}} \, \sum_{n=0}^\infty \left(\frac{ \sigma}{\abs{w}}\right)^n\\
&\leq  
\frac{3}{1-\sigma}   +  \frac{3\,C}{1+2\sigma} \, \sum_{n=0}^\infty \left(\frac{ 3\,\sigma}{1+2\sigma}\right)^n = 
  \frac{3+3\,C}{1- \sigma} =:C_1\;. 
\end{align*}
Therefore, applying Lemma~ \ref{inverses} to 
$w\,I-Q_{z}$ and $w\,I-Q_{0}$,
item (1) holds with $C_2:= \frac{C_1}{1-C_1\,C_0\,b_0}$, while (2) holds with $C_3:=   \frac{C_1^2\,C_0}{1-C_1\,C_0\,b_0}$.  Of course we have to pick $0<b_0<b$ small enough to make sure  the denominators in constants $C_2$ and $C_3$ are both positive.
For the remaining operators
$P_{z}$, $L_{z}$ and $N_{z}$ we use the integral formulas
~\eqref{Pz},~\eqref{Lz} and~\eqref{Nz} to reduce to the previous case, using the same constants $C_2$ and $C_3$
as before.
\end{proof}

\begin{proof}[Proof of Proposition~\ref{abstr:qcs}] 
By Lemma~\ref{spectral:bounds} for all $\vert z\vert <b$ and
$w\notin {\rm int}(\Gamma_0)\cup {\rm int}(\Gamma_1)$, the operator norm $\norm{R_{z}(w)}$ is
uniformly bounded. This implies that the spectrum $\Sigma_{z}$ of $Q_{K, z\,\xi}$ is contained in
${\rm int}(\Gamma_0)\cup {\rm int}(\Gamma_1)$, and hence we can write
$\Sigma_{z}=\Sigma_{z}^0\cup \Sigma_{z}^1$ 
with $ \Sigma_{z}^i\subset {\rm int}(\Gamma_i)$, for $i=0,1$.
By the spectral theory  of bounded operators on Banach spaces,  see  for instance chapter IX in~\cite{RN}, if we denote by $H_{z}$ and $E_{z}$ the subspaces
of  $\B_\alpha$, respectively associated to the spectrum 
components $\Sigma_{z}^0$ and $\Sigma_{z}^1$, then for all 
$z\in\mathbb{D}_{b_0}$, with $b_0>0$ small enough,
\begin{enumerate}
\item[(a)] the operators
$Q_{z}$, $P_{z}$, $L_{z}$ and $N_{z}$ commute,
\item[(a)] $L_{z} f= Q_{z} f\in E_{z}$, for all $f\in E_{z}$,
\item[(b)] $N_{z} f= Q_{z} f\in H_{z}$, for all $f\in H_{z}$,
\item[(c)] $Q_{z}= L_{z} + N_{z}$,
\item[(d)] $\B_\alpha=E_{z}\oplus H_{z}$,
\item[(e)] $P_{z}$ is the projection to $E_{z}$ parallel to $H_{z}$.
\end{enumerate} 
For $z=0$, the condition (A2) implies that the operator
 $Q_{0}\vert_{\B_\alpha}$ is quasi-compact and simple,  with spectrum $\Sigma_{0}^0\subset \mathbb{D}_\sigma$
and $\Sigma_{0}^1=\{1\}$.
Since $1$ is a simple eigenvalue,  $E_{0}=\langle \one\rangle$ is the space of constant functions.
We must have
  $H_{0}=\{\, f\in\B_\alpha\,:\, \int f\,d\mu=0\,\}$ because the operator $Q_{0}$ acts invariantly  on this space, as a contraction with spectral radius $\leq  \sigma$.
Thus for all $f\in\B_\alpha$,
$P_{0} f= (\int f\,d\mu)\,\one$ and
$N_{0} f= Q_{0} f - (\int f\,d\mu)\,\one$.
Since $1$ is a simple eigenvalue of $Q_{0}$,
a continuity argument implies that $\Sigma_{z}^1$ is a singleton, i.e., $\Sigma_{z}^1=\{\lambda(z)\}$, for all $z\in\mathbb{D}_{b}$.
It follows easily that 
$\dim (E_{z})=1$, and $\lambda(z)=\langle L_{z} \one,\mu\rangle /\langle P_{z} \one,\mu\rangle$.
By perturbation theory, and Proposition~\ref{Qz:analytic},  the function 
$\lambda:\mathbb{D}_{b_0}\to \C$ is analytic. Hence,
to finish the proof of Proposition~\ref{abstr:qcs}, it is now enough
to establish  items (6)-(10).

Take $0<b_0<b$ according to Lemma~\ref{spectral:bounds}.
Fixing  a reference probability measure $\mu_0$ on $\Sigma$,
we can write, for all $z\in\mathbb{D}_{b}$,
\begin{equation}\label{lambda gamma}
\lambda_{K,\mu,\xi}(z) = \frac{\langle L_{K, z\,\xi}\one, \mu_0  \rangle}{\langle P_{K, z\,\xi}\one, \mu_0  \rangle}\;.
\end{equation}
Notice that by Lemma~\ref{spectral:bounds}, for all $(K,\mu,\xi)\in\Xfrk$,
$$ \langle P_{K, z\,\xi}\one,\mu_0 \rangle \geq 1-\norm{P_{K, z\,\xi}\one -P_{K,0}\one}_\alpha
\geq 1-C_0\,b_0\;.$$
Hence, for all  $z\in\mathbb{D}_{b_0}$,
\begin{align*}
\abs{\lambda_{K,\mu,\xi}(z)-1} &\leq \abs { 
\frac{\langle L_{K, z\,\xi}\one,\mu_0 \rangle }{ \langle P_{K, z\,\xi}\one,\mu_0 \rangle } -
\frac{\langle L_{K,0}\one,\mu_0 \rangle }{ \langle P_{K,0}\one,\mu_0 \rangle } } \\
&\leq  
\frac{\abs{ \langle L_{K, z\,\xi}\one-L_{K,0}\one ,\mu_0 \rangle} }{ 1-C_0\,b_0 }  +
\frac{C_0\, \abs{ \langle P_{K, z\,\xi}\one-P_{K,0}\one,\mu_0 \rangle} }{ ( 1-C_0\,b_0)^2 } \\
&\leq  
\frac{ C_0\,b_0 }{ 1-C_0\,b_0 }  +
\frac{C_0^2\,b_0  }{ ( 1-C_0\,b_0)^2 } =O(b_0)\;.
\end{align*}
Thus, given $\varepsilon>0$ we can make
$b_0>0$ small enough so that  for all $(K,\mu,\xi)\in\Xfrk$,
 and all $z\in\mathbb{D}_{b_0}$, $\abs{\lambda_{K,\mu,\xi}(z)-1}<\varepsilon$. This implies (6).
 
To prove (7), choose $p\in\N$ such that
$C\,\sigma^p\leq  (\sigma+\frac{\varepsilon}{2})^p$, and  make $b_0>0$ small enough so that
$$ p\, C_0^{p}\,b_0 < (\sigma+ \varepsilon)^p - (\sigma+\frac{\varepsilon}{2})^p=O(\varepsilon)\;.$$
We have then 
\begin{align*}
\norm{N_{z}^p} &\leq \norm{N_{0}^p}  + \norm{N_{z}^p- N_{0}^p} \\
&\leq C\,\sigma^p  + p\,C_0^{p-1}\,\norm{N_{z}- N_{0}} 
\leq C\,\sigma^p  + p\,C_0^{p}\,b_0\\
& \leq C\,\sigma^p  + (\sigma+ \varepsilon)^p - (\sigma+\frac{\varepsilon}{2})^p <  (\sigma+ \varepsilon)^p\;.
\end{align*}
It follows that for all $n\in\N$,
$\norm{N_{z}^n}\leq C_0^{p}\, (\sigma+ \varepsilon)^n$. This proves (7) with $C'= C_0^{p}$.

Items (8) and (9) follow from Lemma~\ref{spectral:bounds}.

To prove item (10), we claim that for all
$(K_1,\mu_1,\xi_1), (K_2,\mu_2,\xi_2)\in\Xfrk$, $z\in\mathbb{D}_{b_0}$, 
$2 \alpha_1\leq \alpha\leq \alpha_0$, and $f\in\B_\alpha$,
\begin{equation} \label{Lip:Q}
v_{\alpha\over 2}(Q_{K_1, z\xi_1 } f - Q_{K_2, z\xi_2} f )
\lesssim\, \norm{f}_\alpha\, \dist\left((K_1,\mu_1,\xi_1), (K_2,\mu_2,\xi_2)\right)^{\theta\over 2}\;. 
\end{equation}
In fact by (B7), (B2) and (A4), we have
\begin{align*}
v_{\alpha\over 2}(Q_{ K_1, z\xi_1 } f - Q_{K_2, z\xi_2} f) &\leq v_0(Q_{K_1, z\xi_1} f - Q_{K_2, z\xi_2} f)^{1\over 2}\, v_\alpha(Q_{K_1, z\xi_1} f - Q_{K_2, z\xi_2} f)^{1\over 2}\\
& \lesssim 
\norm{Q_{K_1, z\xi_1} f - Q_{K_2, z\xi_2} f}_\infty^{1\over 2}\, v_\alpha(Q_{K_1, z\xi_1} f - Q_{K_2, z\xi_2} f)^{1\over 2}\\
& \lesssim 
 \norm{f}_\alpha \, \dist\left((K_1,\mu_1,\xi_1), (K_2,\mu_2,\xi_2)\right)^{\theta \over 2}   \;. 
\end{align*}
Equation~\eqref{Lip:Q} implies, for all $(K_1,\mu_1,\xi_1)$, $(K_2,\mu_2,\xi_2)$, $z$, $\alpha$
and $f$ as above, and all
$w \notin {\rm int}(\Gamma_0)\cup {\rm int}(\Gamma_1)$,
\begin{equation}\label{Lip:R}
v_{\alpha\over 2}(R_{K_1,z\xi_1}(w) f - R_{K_2, z\xi_2}(w) f )
\lesssim\, \norm{f}_\alpha\, \dist\left((K_1,\mu_1,\xi_1), (K_2,\mu_2,\xi_2)\right)^{\theta\over 2}\;. 
\end{equation}
This follows from~\eqref{Lip:Q}, Lemma~\ref{spectral:bounds},
and the algebraic relation
$$ R_{K_1, z\xi_1} (w) - R_{K_2, z\xi_2} (w)  =
- R_{K_1, z\xi_1} (w) \circ (Q_{K_1, z\xi_1} - Q_{K_2, z\xi_2} )\circ  R_{K_2, z\xi_2} (w)\;. $$
Thus, integrating~\eqref{Pz} and ~\eqref{Lz}, we obtain
\begin{align*} 
\norm{P_{K_1, z\xi_1}  f - P_{K_2, z\xi_2}  f}_{\alpha\over 2} &\lesssim \,\norm{f}_\alpha\,\dist\left( (K_1,\mu_1,\xi_1), (K_2,\mu_2,\xi_2)\right)^{\theta\over 2}\;, \\
\norm{L_{K_1, z\xi_1}  f - L_{K_2, z\xi_2}  f}_{\alpha\over 2} &\lesssim \,\norm{f}_\alpha\,\dist\left( (K_1,\mu_1,\xi_1), (K_2,\mu_2,\xi_2)\right)^{\theta\over 2}\;.
\end{align*}
Finally, (10) follows from the previous
inequalities and~\eqref{lambda gamma}.
\end{proof}

\begin{remark}\label{Holder continuity}
The condition $\alpha_1<\frac{\alpha_0}{2}$ and the assumption (A4)
are only needed to prove (10). 
\end{remark}

\subsection{An abstract theorem}
\label{random_abs_theor}

In this subsection we state and prove an abstract LDT theorem.

Let  $(\B_\alpha,\norm{\cdot})_{\alpha\in [0,1]}$ be a scale of Banach algebras satisfying (B1)-(B7).
Assume $\Xfrk$ is a metric space of observed Markov systems for which assumptions (A1)-(A4) hold.
Take $0<b_0<b$ according to proposition~\ref{abstr:qcs}.

Given $(K,\mu,\xi)\in \Xfrk$, let $c_{K,\xi}(z):=\log\lambda_{K,\xi}(z)$,
where  $\lambda_{K,\xi}(z)$ denotes the maximal eigenvalue of $Q_{K, t\xi}$.

\begin{theorem}\label{ALDE:uniform}
Given  $(K_0,\mu_0, \xi_0) \in\Xfrk$ and $\ctwo >(c_{K_0,\xi_0})''(0)$, there exist  a neighbourhood $\V$ of $(K_0,\mu_0,\xi_0)\in \mathscr{X}$,
  $C>0$ and $\varepsilon_0>0$
such that for all $(K,\mu,\xi)\in \V$, $0<\varepsilon<\varepsilon_0$, $x\in\Sigma$ and  $n\in\N$,
\begin{equation}\label{abst:ldt}
\Pp_x^+\left[ \, \abs{\frac{1}{n}\,S_n(\xi)- \EE_\mu(\xi)} \geq \varepsilon \,\right]\leq C\,e^{ - {\varepsilon^2\over {2\, \ctwo}}\, n  } \;. 
\end{equation}
\end{theorem}

%\begin{remark}\label{rmk:uniformity}
%The proof of this theorem shows that conclusion ~\eqref{abst:ldt} holds 
%for any $(K,\xi)\in \mathscr{X}$ such that  $\ctwo >(c_{K,\xi})''(0)$.
%\end{remark}
%
%
%\begin{corollary}\label{ALDE:xi in B}
%Assume  $\Xfrk_{K,L}^\alpha$ satisfies (A2).
% 
%
%Given $\xi_0 \in\Xfrk_{K,L}^\alpha$ with $\ctwo >(c_{K,\xi_0})''(0)$,  there exist  a neighborhood $\V$ of $\xi_0$ in $\Xfrk_{K,L}^\alpha$,
%  $C>0$ and $\varepsilon_0>0$
%such that for all $\xi\in \V$, $0<\varepsilon<\varepsilon_0$, $x\in\Sigma$ and  $n\in\N$,
%the inequality ~\eqref{abst:ldt} holds.
%%$$ \Pp_x\left[ \, \abs{\frac{1}{n}\,S_n(\xi)- \EE_\mu(\xi)} \geq \varepsilon \,\right]\leq C\,e^{ - {\varepsilon^2\over {2\,\ctwo}}\, n  } \;. $$
%\end{corollary}
%
%\begin{proof} Combine  proposition~\ref{A2-A3 for xi in B} with theorem~\ref{ALDE:uniform}.
%\end{proof}

\begin{remark}\label{average:LDE}
Averaging  in $x$,  w.r.t. $\mu$, the probabilities in theorem ~\ref{ALDE:uniform},   we get 
for all  $0<\varepsilon<\varepsilon_0$, $(K,\mu,\xi) \in \V$   and $n\in\N$, 
$$ \Pp_{\mu}^+\left[ \, \abs{\frac{1}{n}\,S_n(\xi)- \EE_\mu(\xi)} \geq \varepsilon \,\right]\leq C\,e^{ - {\varepsilon^2\over 2 \ctwo}\, n  } \;. $$
\end{remark}

\begin{lemma}\label{ESn:Qn} 
For all $(K,\mu,\xi)\in\Xfrk$, $n\in\N$, $z\in\mathbb{D}_{b_0}$  and  $x\in\Sigma$,
$$ ((Q_{K,z\xi})^n \one )(x)= \EE_x\left[ e^{z\, S_n(\xi)} \right] =
\int_{X^+}  e^{z\, S_n(\xi) } \, d\Pp_x^+ \;. $$
In particular,  for all $z\in \mathbb{D}_{b_0}$,
$$ \EE_{\mu}((Q_{K,z\xi})^n \one )= \EE_{\mu} \left[ e^{z\, S_n(\xi)} \right] \;. $$
\end{lemma}

\begin{proof}  In fact,
$$ ((Q_{K,z\xi})^n \one )(x_0)  =
\int_{\Sigma^n} e^{z\,\sum_{j=1}^{n} \xi(x_j)}\,
\prod_{j=0}^{n-1} K(x_j,dx_{j+1}) = \EE_{x_0}\left[ e^{z\, S_n(\xi)} \right] \;. $$
Averaging this relation in   $x_0$ w.r.t. $\mu$ we derive
the second identity. \qed
\end{proof}

Next proposition shows that
$c_{K,\xi}(z)$ is a limit cumulant generating function of the process $\{S_n(\xi)\}_{n\geq 0}$. Moreover it
says that the parameters
$C$ and $\delta_n$ in definition~\ref{def ucgf} can be chosen uniformly in $\Xfrk$.

\begin{proposition}\label{ESn:lambda} 
There  exist $C_1>0$ and a sequence $\delta_n$ converging geometrically to $0$ 
such that for all $(K,\mu,\xi)\in\Xfrk$, $z\in \mathbb{D}_{b_0}(0)$,  $x\in\Sigma$ and $n\in\N$
$$ \abs{  n\,\log \lambda_{K,\xi}(z) - \log \EE_x \left[ e^{z\, S_n(\xi)} \right]  } \leq   C_1\,\vert z \vert   + \delta_n  \;. $$
\end{proposition}

\begin{proof} We will use the notation of Proposition~\ref{abstr:qcs},
choosing $\varepsilon>0$ small enough so that  $\sigma+\varepsilon<1-\varepsilon$.
By Lemma~ \ref{ESn:Qn}, $(Q_{z}^n\one)(x) = \EE_x\left[ e^{z\,S_n(\xi)}\right]$. 
By Lemma~\ref{spectral:bounds} there exists $B>0$ such that
for all $z\in \mathbb{D}_{b_0}(0)$,\,
$\norm{P_{z}-I}_\alpha\leq B\,\vert z \vert $.
Hence
\begin{align*}
\abs{\EE_x\left[ e^{z\,S_n(\xi)}\right] - \lambda_{K,\xi}(z)^n} & \leq 
 \abs{ (Q_{z}^n \one)(x) -\lambda_{K,\xi}(z)^n  }   \\
&\leq  \norm{    Q_{z}^n\one   -\lambda_{K,\xi}(z)^n P_{z} \one }_\alpha  +
\lambda_{K,\xi}(z)^n\,\norm{\one -P_{z} \one }_\alpha \\
&=   \norm{ N_{z}^n \,\one  }_\alpha  + \lambda_{K,\xi}(z)^n\,\norm{\one -P_{z} \one }_\alpha\\
&\leq  C\,(\sigma+\varepsilon)^n  + B\,\vert z \vert \,\lambda_{K,\xi}(z)^n  \;.
\end{align*}
Thus
\begin{align*}
\abs{  \log \EE_x\left[ e^ {t S_n(\xi)}\right]-n\,\log\lambda_{K,\xi}(z) } &=
\abs{  \log \EE_x\left[ e^ {t S_n(\xi)}\right] -  \log \lambda_{K,\xi}(z)^n }\\
&\leq \frac{ \abs{  \EE_x\left[ e^ {t S_n(\xi)}\right] -  \lambda_{K,\xi}(z)^n } }{\min\{ \lambda_{K,\xi}(z)^n,  \EE_x\left[ e^ {t S_n(\xi)}\right] \}  }\\
&\leq  \frac{  B\,\vert z \vert \,\lambda_{K,\xi}(z)^n + C\,(\sigma+\varepsilon)^n }{(1- B\,\vert z \vert )\,\lambda_{K,\xi}(z)^n- C\,(\sigma+\varepsilon)^n  }	\\
&\leq  \frac{   B\,\vert z \vert   + \delta_n }{ 1- B\,\vert z \vert -\delta_n  }	
\leq 2\,(  B\,\vert z \vert   + \delta_n)\;,
\end{align*}
where  $\delta_n:=  C\,\frac{(\sigma+\varepsilon)^n}{\lambda_{K,\xi}(z)^n}\leq C\,\left(\frac{\sigma+\varepsilon}{1-\varepsilon} \right)^n$ converges geometrically to zero. 
\qed
\end{proof}

\begin{proof}[theorem~\ref{ALDE:uniform}]
Combine Proposition~\ref{ESn:lambda} 
with Corollary~\ref{coro uniform LDT estimates}.
\qed
\end{proof}

\bigskip

\section{The proof of LDT estimates}
\label{random_proof}

We prove here the base-LDT and uniform fiber-LDT estimates for irreducible cocycles over mixing
Markov shifts. These results follow from the abstract Theorem~\ref{ALDE:uniform}.

\subsection{Base LDT estimates}
\label{random_base_ldt}
To deduce theorem~\ref{Base:LDT} from theorem~\ref{ALDE:uniform}
we  specify the data $(\B_\alpha,\norm{\cdot}_\alpha)$
and $\Xfrk$, and check the validity of the assumptions (B1)-(B7) and (A1)-(A4).

%%% Markov Systems
Consider a strongly mixing Mar\-kov system  $(K,\mu)$  on the compact metric space $\Sigma$. 
Let  $X^-= \Sigma^{\Z^-_0}$  be the space of sequences in $\Sigma$ indexed in the set $\Z^-_0$ of non-positive integers.
Since $\Z_0^-$ is countable, the product $X^-$ is a compact metrizable topological space. We denote by $\Fcal$ its Borel $\sigma$-field.
The kernel $K$ on $\Sigma$ induces another  Markov kernel $\Ktil$  on $X^-$   defined by
$$ \Ktil_{(\,\ldots, \, x_{-1}, x_{0})} :=\int_\Sigma \delta_{(\,\ldots,  \, x_{-1}, x_{0}, x_{1})}\, K(x_0,dx_1)\;. $$
Let $\Pp^-_\mu$ denote the Kolmogorov extension of $(K,\mu)$, which is also the unique $\Ktil$-stationary measure. 
Theorem~\ref{ALDE:uniform} will be applied to the Markov system
$( \Ktil, \Pp^-_\mu)$.

Consider the spaces $\Hscr_\alpha(X^-)$ 
introduced in definition~\eqref{def Hscr alpha X minus}.
Its functions
can be regarded as measurable functions on $X^-$.
They form the scale of Banach algebras satisfying (B1)-(B7).
See proposition~\ref{Holder:B1:B7}. 
The metric space $(X^-, \dtil)$ has diameter $1$ but is not compact, as noticed after the definition~\eqref{dtil} of the distance $\dtil$.
Hence, formally, the claim above is not a direct consequence of proposition~\ref{Holder:B1:B7}. 
Properties  (B1), (B3) and (B4) follow from
proposition~\ref{algebra:Ha}.
For $\alpha=0$, the seminorm $v_0$ measures the variation of $f$. Hence $\Hscr_0(X)=L^\infty(X)$,
while the norm $\norm{\cdot}_0$ is equivalent to $\norm{\cdot}_\infty$. This proves  (B2).
The remaining properties, (B5)-(B7), can be proved as in proposition~\ref{Holder:B1:B7}.

\medskip

 Fix $0<\alpha_0\leq 1$ and $0<L<+\infty$ and consider the space
  $\Xfrk$  of observed Markov systems
 $(\Ktil,\Pp_\mu^-,\xi)$ over the fixed Markov system $(\Ktil,\Pp_\mu^-)$,   with $\xi\in \Hscr_{\alpha_0}(X^-)$
 and $\norm{\xi}_{\alpha_0}\leq L$. This space is  identified with a subspace of $\Hscr_{\alpha_0}(X^-)$, and endowed with the corresponding norm distance.
 
\medskip

The kernel $\Ktil$ determines the Markov operator
$Q_{\Ktil}:L^\infty(X^-)\to L^\infty(X^-)$, 
$$ (Q_{\Ktil} f)(\,\ldots, \, x_{-1}, x_{0})  :=\int_\Sigma f(\,\ldots,  \, x_{-1}, x_{0}, x_{1}) \, K(x_0,dx_1)\;. $$ 
This operator  acts continuously on $\Hscr_a(X^-)$.

\begin{proposition} \label{Vnorm}  For all $f\in \Hscr_\alpha(X^-)$ and $n\in\N$,
\begin{enumerate}
\item[(1)] $\displaystyle \norm{(Q_{\Ktil})^n  f}_\infty   \leq \norm{f}_\infty$,
\item[(2)]  $\displaystyle v_\alpha((Q_{\Ktil})^n  f)  \leq \max\{ 2\,\norm{(Q_{\Ktil})^n f}_\infty, \, 2^{-n\,\alpha} v_\alpha (f)\} $.
\end{enumerate}
\end{proposition}

\begin{proof} We shall write $Q= Q_{\Ktil}$.
Since $\int_\Sigma K(x_0,dx_1) =1$, the first inequality follows.
For the second, notice that if $k\geq 1$
then $v_k(Q^n f)\leq v_{k+n}(f)$. 
Indeed, for $x=(x_n)_{n\leq 0}$ and $x'=(x_n')_{n\leq 0}$ in  $X^-$
such that $\dtil(x,x')\leq 2^{-k}$ with $k\geq 1$, we have $x_0=x_0'$. Thus
\begin{align*}
&\abs{(Q^n f)(\,\ldots, x_{-1}, x_0)-(Q^n f)(\,\ldots, x_{-1}', x_0') }  \\
& \quad \leq \int_{\Sigma^n} \abs{f(\,\ldots, x_0, x_1,\ldots, x_n) -f(\,\ldots, x_0', x_1,\ldots, x_n) } \prod_{j=0}^{n-1} K(x_j,dx_{j+1}) \\
& \quad \leq v_{k+n}(f) \int_{\Sigma^n} \prod_{j=0}^{n-1} K(x_j,dx_{j+1})
=v_{k+n}(f)\;,
\end{align*}
and taking the sup in $x,x'\in X^-$ such that $\dtil(x,x')\leq 2^{-k}$, the inequality  $v_k(Q^n f)\leq v_{k+n}(f)$ follows.
Hence, for $k\geq 1$,
$$ 2^{\alpha  k} v_k(Q^n f)= 2^{-n\,\alpha} ( 2^{\alpha (k+n)} v_{k+n}(f))\leq 2^{-n\,\alpha} v_\alpha(f) \;. $$
For $k=0$ notice that $v_0(Q^n f)$ is the variation of $Q^n f$. Thus  $ v_0(Q^n f) \leq 2\, \norm{Q^n f}_\infty$.
Taking the sup in $k\in\N$, item (2) follows.
\end{proof}

Next proposition shows that $\Xfrk$ satisfies
 (A2) with range $[\alpha_1,\alpha]$ for any given $0<\alpha_1\leq\alpha$.
The setting constants $C>0$ and $0<\sigma<1$ depend  on the number  $\alpha_1$.

\begin{proposition}\label{Base:QCS}
If $(K,\mu)$ is strongly mixing, then  given $0<\alpha_1< \alpha_0$
 there are constants $C>0$ and $0<\sigma<1$ such that for all $\alpha_1\leq\alpha\leq \alpha_0$,  $Q_{\Ktil}:\Hscr_\alpha(X^-)\to \Hscr_\alpha(X^-)$ is quasi-compact and simple with spectral constants $C$ and $\sigma$, i.e.,  for all   $f\in \Hscr_\alpha(X^-)$,
$$ \norm{(Q_{\Ktil})^n f -\langle f,\Pp^-_\mu \rangle \one }_\alpha\leq C\,\sigma^n\, \norm{f}_\alpha\;. $$ 
\end{proposition}

\begin{proof}
Given a function $f\in \Hscr_\alpha(X^-)$, denote by $f_k:X^-\to\C$ the following function
$$ f_k(\,\ldots, x_0):= \int_{X^-} f(\,\ldots, x_{-k},\ldots , x_0)\, d\Pp_\mu^-(\,\ldots, x_{-k}) \;. $$
Note that if $\FF_k^-$ is the sub $\sigma$-field of $\FF^-$ generated
by the cylinders in the coordinates $x_{-k+1},\ldots, x_{-1}, x_0$, we have $f_k=\EE_\mu^-(f\vert \FF_k^-)$, and in particular 
$\EE_\mu^-(f_k)=\EE_\mu^-(f)$, for all $k\in\N$.
By definition of $f_k$,
\begin{equation}\label{fk}
\norm{Q^n (f- f_k)}_\infty\leq \norm{f-f_k}_\infty \leq v_k(f)\leq 2^{-\alpha k} v_\alpha(f)\;.
\end{equation}  
Because $(K,\mu)$ is strongly mixing, there are constants $C>0$ and $0<\rho<1$ such that for any function $h\in L^\infty(\Sigma)$ with   $\int_\Sigma h\, d\mu=0$,  
$$ \abs{ \int_\Sigma h(y) K^n(x,dy) }\leq C\,\rho^n \,\norm{h}_\infty \;.$$
Now, if $h\in L^\infty(X^-)$ is a function with zero average,
i.e., $\EE_\mu^-(h)=0$,  which depends only on the first coordinate $x_0$, then $Q^n h$ also depends only on the first coordinate, and  is given by
$$ (Q^n h)(\,\ldots, x_0):= \int_\Sigma h(y) K^n(x_0,dy)\;.$$
Hence 
\begin{equation}\label{Qnh}
\norm{Q^n h}_\infty \leq C\,\rho^n\, \norm{h}_\infty\;. 
\end{equation}
We claim that $h= Q^k(f_k- \EE_\mu^ -(f) \,\one)$ is a function with zero average that depends only on the first coordinate.
The first part of claim follows because $Q$ preserves averages
and, as remarked above, $\EE_\mu^-(f_k)=\EE_\mu^-(f)$.
For the second part notice two things: first $Q$ `preserves' functions that depend only on the first coordinate $x_0$; second, $Q$ maps a 
function $f$ that depends only on the coordinates
$x_{-k},\ldots, x_{-1}, x_0$ to a function that  depends only on the coordinates $x_{-k+1},\ldots, x_{-1}, x_0$, in other words $Q f$ looses dependence in $x_{-k}$.
Therefore, from~ \eqref{Qnh}
\begin{align}\label{Qnfk}
\norm{Q^n(f_k - \EE_\mu^-(f)\,\one) }_\infty
&= \norm{ Q^{n-k} h }_\infty  \leq C\,\rho^{n-k} \,\norm{h}_\infty  \\
&\leq C\,\rho^{n-k} \,\norm{Q^k(f_k- \EE_\mu^ -(f) \,\one)}_\infty \nonumber \\
&\leq C\,\rho^{n-k} \,\norm{ f_k- \EE_\mu^ -(f) \,\one}_\infty  
 \leq 2 C\, \rho^{n-k} \,\norm{f}_\infty  \nonumber
\end{align}
Setting  $\sigma= \max\{ 2^{-{\alpha_1\over 2}}, \sqrt{\rho}\}$  
we have  $0<\sigma<1$.
From the inequalities~\eqref{fk} and~\eqref{Qnfk},  with $k=n/2$, we have
\begin{align*}
\norm{ Q^n f - \EE_\mu^-(f)\,\one }_\infty  &\leq
\norm{ Q^n (f-f_k) }_\infty +
\norm{ Q^n (f_k - \EE_\mu^-(f)\,\one) }_\infty\\
&\leq 2^{-\alpha \frac{n}{2}} v_\alpha(f)+ 2 C \rho^{\frac{n}{2}} \norm{f}_\infty \\
&\leq \sigma^n v_\alpha(f) + 2 C \sigma^n \norm{f}_\infty\;.
\end{align*}
On the other hand, by item (2) of Proposition~ \ref{Vnorm},
\begin{align*}
v_\alpha(Q^n f - \EE_\mu^-(f)\,\one) &= v_\alpha(Q^n (f - \EE_\mu^-(f)\,\one) )\\
&\leq  \max\{\,\norm{ Q^n f - \EE_\mu^-(f)\,\one }_\infty,\,
2^{-n\,\alpha}\,v_\alpha(f)\,\}\\
&\leq  \max\{\,\sigma^n v_\alpha(f) + 2 C \sigma^n \norm{f}_\infty,\,
\sigma^{2n}\,v_\alpha(f)\,\}\\
&\leq \sigma^n v_a(f) + 2 C \sigma^n \norm{f}_\infty\;.
\end{align*}
Thus, for all $f\in \Hscr_\alpha(X^-)$,
\begin{align*}
\norm{ Q^n f - \EE_\mu^-(f)\,\one }_\alpha\leq 4 C \sigma^n \norm{f}_\alpha\;,
\end{align*}
which proves the proposition.
\end{proof}

\medskip

\subsection{Fiber LDT estimates}
\label{random_fiber_ldt}
In this subsection we use theorem~\ref{ALDE:uniform} to establish the fiber LDT theorem~\ref{Fiber:LDT}.
First we  specify the data  $(\B_\alpha,\norm{\cdot}_\alpha)$
and the metric space $\Xfrk$. Then we check that assumptions (B1)-(B7) and (A1)-(A4) hold up.

% Markov systems   
Consider the space $\randcocycles{\infty}{m}$ of random cocycles over a Markov system $(K,\mu)$.
% the kernel
For each cocycle $A\in\randcocycles{\infty}{m}$ we define 
a Markov kernel on  $\Sigma\times\Sigma\times\Proj$  by
\begin{equation}\label{KhatA:2}
K_A (x,y,p):= \int_\Sigma \delta_{(y,z,A(y,z)p)}\, K(y,dz) \;. 
\end{equation}
% the stationary measure
We will see that (c.f. corollary~\ref{uniqueness of stationary measure}), under the assumptions of theorem~\ref{Fiber:LDT}, this kernel admits a unique $K_A$-stationary probability measure $\mu_A$ in $\Sigma\times\Sigma\times\Proj$. 
 
% the observables
For each  $A\in\randcocycles{\infty}{m}$ consider the observable   $\xi_A:\Sigma\times\Sigma\times\Proj\to \R$ 
\begin{equation}\label{fiber:observable}
\xi_A(x,y,p):= \log \norm{ A(x,y)\,p}\;. 
\end{equation}

% the space of observables
We can now introduce the metric space of observed Markov systems 

$$ \Xfrk:=\{ \,(K_A,\mu_A, \pm\xi_A)\,\colon\, A\in\randcocycles{\infty}{m}, \; A \; \text{ irreducible}, \; L_1(A)>L_2(A) \,\} \;.$$
This space is identified with a subspace of $\randcocycles{\infty}{m}$, and 
% its distance
endowed with the distance
$$ \dist\left( (K_A,\mu_A,\xi_A), (K_B,\mu_B,\xi_B)\right) := d_\infty(A,B) \;. $$

Next we define the scale of Banach algebras.
Consider the following projective distance
(see~\cite[formula (1.3)]{LEbook-chap2})
 $$ \delta(\hat p, \hat q):=\frac{ \norm{p\wedge q}}{\norm{p} \norm{q} } \;, $$ 
where  $p\in \hat p$ and $q\in \hat q$.
% Define the space
Given $0\leq \alpha\leq 1$ and $f\in L^\infty(\Sigma\times\Sigma\times\Proj)$, let 
\begin{align}
\norm{f}_\alpha &:= v_\alpha(f) + \norm{f}_\infty \;, \label{norm:a} \\
v_\alpha(f) &:=\sup_{\substack{ x,y,\in\Sigma\\ p\neq q}}  \frac{\abs{f(x,y,p)-f(x,y,q)}}{\delta(p,q)^\alpha} \;. \label{V:a}
\end{align}

\begin{definition}
\label{def H alpha Sigma x Sigma x Proj}
Consider the normed space  $\Hscr_\alpha(\Sigma\times\Sigma\times\Proj)$  of all  functions $f\in L^\infty(\Sigma\times\Sigma\times\Proj)$ such that 
$v_\alpha(f)<+\infty$, endowed with the norm~\eqref{norm:a}.
\end{definition}

%% proposition: scale of Banach algebras satisfying (B1)_(B7)
\begin{proposition}\label{Holder:fiber:B1:B7}
The family of spaces  
$\Hscr_\alpha(\Sigma\times\Sigma\times\Proj)$ is a scale of Banach algebras satisfying (B1)-(B7). 
\end{proposition}

\begin{proof}
(B1) holds by definition of the H\"older norm $\norm{\cdot}_\alpha$.
For (B2) notice that $v_0(f)$ measures the maximum oscillation of $f$ on the projective fibers, and hence 
$v_0(f)\leq 2\,\norm{f}_\infty$.
Property (B3) is obvious.
Assumption (B4) is a consequence of the inequality
$$v_\alpha(f\,g)\leq \norm{f}_\infty\,v_\alpha(g) + \norm{g}_\infty\,v_\alpha(f),\quad f,g\in L^\infty(\Sigma)\;. $$

The monotonicity properties (B5) and (B6) are straightforward to check.
The assumption (B7) follows from the convexity of the function  $\alpha\mapsto \log v_\alpha(f)$, whose proof is analogous to that of proposition~\ref{Holder:B1:B7}.
\end{proof}

\begin{definition}
\label{def H alpha Sigma x Proj}
We define $\Hscr_\alpha( \Sigma\times\Proj)$
to be the subspace of functions $f(x,y,p)$ in $\Hscr_\alpha(\Sigma\times\Sigma\times\Proj)$ 
that do not depend on the first coordinate $x$.
\end{definition}
This subspace is clearly a closed sub-algebra of  $\Hscr_\alpha(\Sigma\times\Sigma\times\Proj)$.
Therefore, 
\begin{proposition}\label{Holder:fiber:B1:B7}
The family
  $ \Hscr_\alpha( \Sigma\times\Proj)$  is a  scale of Banach sub-algebras  satisfying  (B1)-(B7).
\end{proposition}

%\bigskip

% Markov operator
Given   $A\in\randcocycles{\infty}{m}$, consider the linear transformation
$Q_A:L^\infty(\Sigma\times\Sigma\times\Proj)\to L^\infty(\Sigma\times\Sigma\times\Proj)$ defined by
\begin{equation}\label{QA}
 (Q_A f)(x,y,p):= \int_\Sigma f(y,z, A(y,z) p)\, K(y,dz) \;. 
\end{equation}
This is the Markov operator associated with the kernel ~\eqref{KhatA:2}.

\medskip

%(A1)
Assumption (A1) follows from the definition of $\Xfrk$.

\newcommand{\xim}[1]{\kappa^{#1}_\alpha}
\newcommand{\ximo}[1]{\kappa^{#1}_{\alpha_0}}

Since $(Q_A f)(x,y,p)$ does not depend on the coordinate $x$,
the Markov operator $Q_A$ leaves invariant the subspace of functions
$f(x,y,p)$ that are constant in  $x$. Next, we are going to see that $Q_A$
acts invariantly on the subspace $\Hscr_\alpha(  \Sigma\times\Proj)$.

Given $A\in\randcocycles{\infty}{m}$ and $0<\alpha\leq 1$, define for all $n\in\N$,
\begin{equation} \label{xi:alpha}
\xim{n}(A) := \sup_{x\in\Sigma, p\neq q} \EE_x\left[ \left(
\frac{\delta(\Ahat^{(n)}\,p, \Ahat^{(n)}\,q)}{\delta(p,q)}
\right)^\alpha\right] \in [0,+\infty] 
\end{equation}

%Next lemma describes semigroup properties of the   cocycles in the spaces $\randcocycles{\infty}{m}$ with $0<a\leq \infty$.

\begin{lemma} \label{semigroup} 
Let $A \in \randcocycles{\infty}{m}$ and $n\in\N$.
\begin{enumerate}
\item[(a)]
$\norm{ \Ahat^{(\pm n)}}_\infty \leq \max\{ \norm{A}_\infty, \norm{A^{-1}}_\infty\}^n$.

\item[(b)]  
$ \norm{ \Ahat^{(n)} - B^{(n)}}_\infty  \leq n\,\max\{\norm{A}_\infty, \norm{B}_\infty \}^ {n-1}
\,  \norm{A - B}_\infty$.
\end{enumerate}
\end{lemma}

\begin{proof}
Item (a) is straightforward. To prove (b), we use the formula

\quad
$\displaystyle  \Ahat^{(n)}-B^{(n)} = 
\sum_{j=0}^{n-1}  (\Ahat^{(j)} \circ \transl^{n-j}) (\Ahat\circ \transl^{n-1-j} - B\circ \transl^{n-1-j})\, B^{(n-1-j)}$.\; 
\end{proof}

The following lemma highlights the importance of this quantity.

\begin{lemma} \label{Va:QKa}
Given $A\in\randcocycles{\infty}{m}$, $f\in \Hscr_\alpha(\Sigma\times \Proj)$ and $n\in\N$,
$$ v_\alpha( Q_{A}^n f) \leq \xim{n}(A)\, v_\alpha(f)\;. $$
\end{lemma}

\begin{proof} For any $f\in \Hscr_\alpha(\Sigma\times \Proj)$,
and $(x_0,p)\in\Sigma\times\Proj$,
\begin{align*}
(Q_{A}^n f)(x_0,p)
& = \int_{\Sigma^n} f(x_n, A(x_{n-1},x_n)\ldots A(x_0,x_1) \, p)\,\prod_{j=0}^{n-1} K(x_j,d x_{j+1})\\
&  = \EE_{x_0}\left[ f(e_n,A^{(n)} p)  \right]\;.
\end{align*}
Hence
\begin{align*}
 v_\alpha(  Q_{A}^n f ) & = \sup_{x\in\Sigma,p\neq q}
\frac{\abs{ \EE_x\left[ f(e_n,A^{(n)} p) - f(e_n,A^{(n)} q) \right] } }{\delta(p,q)^\alpha}\\
& \leq \sup_{x\in\Sigma,p\neq q}
\frac{ \EE_x\left[ \abs{f(e_n,A^{(n)} p) - f(e_n,A^{(n)} q)} \right]  }{\delta(p,q)^\alpha}\\
& \leq v_\alpha(f)\, \sup_{x\in\Sigma,p\neq q}
\EE_x\left[ \left(\frac{ \delta( A^{(n)} p,A^{(n)} q) }{\delta(p,q)}\right)^ \alpha \right] \\
& = v_\alpha(f) \, \xim{n}(A)  \;.
\qquad \qquad  \qquad \qquad \qquad \qquad \qquad 
\end{align*}
\end{proof}

%Also important 

\begin{lemma}  \label{xi:sub-multiplicative}
The sequence $\{\xim{n}(A)\}_{n\geq 0}$ is
sub-multiplicative, i.e.,  
$$ \xim{n+\ell}(A)\leq \xim{n}(A)\, \xim{\ell}(A)\quad
\text{ for }\; n,\ell\in\N\;. $$
In particular,
$$ \lim_{n\to+\infty} \xim{n}(A)^{1/n} = \inf\{\,\xim{n}(A)^{1/n}
\, :\, n\in\N\, \} \;. $$
\end{lemma} 

\begin{proof} Let us write  $M_n=\Ahat^{(n)}$. Given $x\in\Sigma$ and $p\neq q$ in $\Proj$,
\begin{align*}
& \EE_x\left[ \left(\frac{\delta(M_{n+m} p, M_{n+m} q)}{\delta(p,q)} \right)^\alpha \right] \leq \\
&\leq   \EE_x\left[ \left(\frac{\delta((M_{n}\circ\transl^m) M_m p, (M_{n}\circ\transl^m) M_m q)}{\delta(M_{m} p, M_{m} q) } \right)^\alpha   \,
  \left(\frac{\delta(M_{m} p, M_{m} q)}{\delta(p,q)} \right)^\alpha \right] \\
&\leq   \xim{m}\, \EE_x\left[ \left(\frac{\delta((M_{n}\circ\transl^m) M_m p, (M_{n}\circ\transl^m) M_m q)}{\delta(M_{m} p, M_{m} q) } \right)^\alpha   \right]  \\
&\leq  \xim{m}\, \sup_{p\neq q} \EE_{K^m(x,\cdot)} \left[ \left(\frac{\delta( M_{n} p, M_{n} q)}{\delta(p,q) } \right)^\alpha   \right]   \leq \xim{m} \, \xim{n}\;, 
\end{align*}
and taking the sup we get $\xim{n+m} \leq \xim{n}\, \xim{m}$.
\end{proof}

These constants become finite provided $\alpha$ is small enough.

\begin{lemma} \label{xi:finite} Given $A\in\randcocycles{\infty}{m}$ and $n\in\N$ for all \, $0<\alpha\leq \frac{1}{4n}$,
$$\xim{n}(A) \leq \max\{\norm{A}_\infty,\norm{A^{-1}}_\infty\}\;.$$
\end{lemma}

\begin{proof} We write as before  $M_n=\Ahat^{(n)}$.
Recall that given $M\in\GL(m,\R)$,  the quantity
$ \ell(M):= \max\{ \log \norm{M},\log \norm{M^{-1}} \}$ is sub-multiplicative, in the sense that for any matrices $M_1,M_2\in \GL(m,\R)$, $\ell(M_1\,M_2)\leq \ell(M_1)+\ell(M_2)$.
By~\cite[Lemma 3.26]{LEbook-chap2}, given $x\in\Sigma$, and $p\neq q$ in $\Proj$,
$$
 \EE_x\left[ \left(
\frac{\delta(M_n\,p, M_n\,q)}{\delta(p,q)}
\right)^\alpha\right] 
 = \EE_x\left[\exp \left(\alpha\,\log 
\frac{\delta(M_n\,p, M_n\,q)}{\delta(p,q)}
\right) \right] \leq \EE_x\left[ e^{4\,\alpha\, \ell(M_n) }\right]  \;.
$$
If $0<\alpha\leq \frac{1}{4n}$, setting  $c:=\max\{ \log \norm{A}_\infty, \log \norm{A^{-1}}_\infty \}$ 
$$\EE_x\left[ e^{4\,\alpha\, \ell(M_n) }\right] \leq e^{4 n \alpha c}\leq e^{c}=\max\{\norm{A}_\infty, \norm{A^{-1}}_\infty\}\;. $$
Hence, taking the sup in $x$  and $p\neq q$  we obtain $\xim{n}\leq \max\{\norm{A}_\infty, \norm{A^{-1}}_\infty\}$.
\end{proof}

By the previous lemmas  the operator
$Q_A$ leaves  the subspace $\Hscr_\alpha(\Sigma\times\Proj)$ invariant, for all small enough $\alpha>0$.
To prove that   $Q_A$ is quasi-compact and simple  all   hypothesis of theorem~\ref{Fiber:LDT} are essential. The irreducibility and gap assumptions are used in the following lemmas.

\begin{lemma}\label{uniform:conv}
Given $A\in \randcocycles{\infty}{m}$ such that $(K_A,\mu_A,\xi_A)\in\Xfrk$ 
 $$ \lim_{n\to+\infty} \frac{1}{n}\,\EE_x(\log \norm{\Ahat^{(n)} \,p}) = L_1(A)\;, $$
with uniform convergence in $(x,p)\in\Sigma\times\Proj$.
\end{lemma}
\begin{proof} See Lemma 3.1 in~\cite{Bou}. 
\end{proof}

\begin{lemma} \label{log:Lipschitz} Given $A\in \randcocycles{\infty}{m}$ such that $(K_A,\mu_A,\xi_A)\in\Xfrk$, there exists $n\in\N$ such that for all $x\in\Sigma$ and $p\neq q$ in $\Proj$,
$$ \EE_x\left[\log \frac{\delta(\Ahat^{(n)} p, \Ahat^{(n)} q)}{\delta(p,q)} \right]\leq -1 \;. $$
\end{lemma} 

\begin{proof}  We write  $M_n=\Ahat^{(n)}$.
Given $x\in\Sigma$ and $p\neq q$ in $\Proj$,
\begin{align*}
& \frac{1}{n}\,\EE_x\left[\log \frac{\delta(M_{n} p, M_{n} q)}{\delta(p,q)} \right] \leq \frac{1}{n}\,\EE_x\left[\log \frac{\norm{ (M_{n} p)\wedge (M_{n} q)}}{\norm{M_{n} p}\,\norm{ M_{n} q}} \,\frac
{\norm{p}\,\norm{q}}{\norm{p \wedge q}} \right]\\
& \leq \frac{1}{n}\,\EE_x\left[\log \frac{\norm{ (M_{n} p)\wedge (M_{n} q)}}{\norm{p \wedge q}}\,\frac
{\norm{p}}{\norm{M_{n} p}}\,\frac
{\norm{q}}{\norm{M_{n} q}}  \right]\\
& \leq \frac{1}{n}\,\EE_x\left[ \log \norm{\wedge_2 \Ahat^{(n)}}\right]
- \frac{1}{n}\,\EE_x\left[ \log  \norm{\Ahat^{(n)} p}  \right]
- \frac{1}{n}\,\EE_x\left[ \log  \norm{\Ahat^{(n)} q}  \right]\;,
\end{align*}
and the right hand side converges to
$L_1 +L_2 - 2\, L_1 = L_2 -  L_1 <0$.
By Lemma~\ref{uniform:conv}, we have
$$ \limsup_{n\to+\infty} \sup_{x\in\Sigma,p\neq q} 
\frac{1}{n}\,\EE_x\left[\log \frac{\delta(M_{n} p, M_{n} q)}{\delta(p,q)} \right] \leq L_2 -  L_1 <0\;. $$
Hence taking $n$ large enough such that
$ n\,(L_2 -  L_1) <-1$ the Lemma follows.
\end{proof}

\begin{proposition}  \label{contr:B}
Given $A\in \randcocycles{\infty}{m}$ such that $(K_A,\mu_A,\xi_A)\in\Xfrk$, there exists  a neighborhood $\V$ of $A$ in $\randcocycles{\infty}{m}$,  and there are constants 
$0<\alpha_1<\frac{\alpha_0}{2}<\alpha_0$, $C>0$ and $0<\sigma<1$  such that
$$ v_\alpha( Q_B^n f ) \leq C\,\sigma^n\, v_\alpha(f)\;,$$
 for all $B\in\V$, $\alpha\in[\alpha_1,\alpha_0]$, 
  $n\in\N$ and $f\in\Hscr_\alpha(\Sigma\times\Proj)$.
\end{proposition}

\begin{proof}  We begin deriving  a modulus of continuity for 
$B\mapsto \xim{n}(B)$.

Fix a neighborhood $\V$ of $A$ in $\randcocycles{\infty}{m}$ such that for all $B\in \V$, $\norm{B}_\infty\leq C$ and
$\norm{B^{-1}}_\infty\leq C$.
By  Lemma~\ref{semigroup}\,  $ \norm{ B^ {(\pm n)} }_\infty  \leq C^n $   for all $B\in\V$ and $n\in\N$.
Thus, by~\cite[Lemma 3.27]{LEbook-chap2}  and Lemma~\ref{semigroup} (b),
there exists a polynomial expression $C(g_1,g_2)$,  with degree $<11$ in the variables $\norm{g_1}$, $\norm{g_2}$, $\norm{g_1^{-1}}$ and $\norm{g_2^{-1}}$, such that
\begin{align*} 
& \abs{ \xim{n}(A) - \xim{n}(B) }  \leq 
\sup_{x\in\Sigma, p\neq q} \EE_x\left[ \left| \left(\frac{\delta(A^{(n)} p, A^{(n)} q) }{\delta(p,q)} \right)^\alpha -
\left(\frac{\delta(B^{(n)} p, B^{(n)} q) }{\delta(p,q)} \right)^\alpha \right| \right] \\
&\leq \alpha \, C(A^{(n)},B^{(n)})\, \norm{A^{(n)} - B^{(n)}}_\infty \leq \alpha\,  C^{11\,n} \, \norm{A^{(n)} - B^{(n)}}_\infty 
\leq \alpha\, n \, C^{12n -1}\, \norm{A-B }_\infty   \;.
\end{align*}
Let  $M_n=\Ahat^{(n)}$.
We claim that  for some   $n_0\in\N$ and $0<\alpha_0\leq 1$  small enough,
   $\ximo{n_0}(A)<1$.
We will make use the following inequality
$$ e^ x \leq 1+x+\frac{x^ 2}{2}\,e^{\abs{x}} \;. $$
Choose ${n_0}\in\N$ as given by Lemma~\ref{log:Lipschitz}.
For all $x\in\Sigma$, $p \neq q$ in $\Proj$,
\begin{align*}
& \EE_x\left[ \left(\frac{\delta(M_{n_0} p, M_{n_0} q)}{\delta(p,q)} \right)^\alpha \right] 
= \EE_x\left[\exp  \left(\alpha\,\log \frac{\delta(M_{n_0} p, M_{n_0} q)}{\delta(p,q)} \right) \right] \\
& \leq \EE_x\left[ 1+\alpha\,\log \frac{\delta(M_{n_0} p, M_{n_0} q)}{\delta(p,q)} + \frac{\alpha^2}{2}\, \log^2 \frac{\delta(M_{n_0} p, M_{n_0} q)}{\delta(p,q)}\, \left(\frac{\delta(M_{n_0} p, M_{n_0} q)}{\delta(p,q)}\right)^\alpha \right]\\
&\leq 1-\alpha + \frac{\alpha^2}{2}\,\EE_x\left[16\,\ell(M_{n_0})^2\,\exp(  \alpha\,\ell(M_{n_0}))\right] 
\leq 1-\alpha + O(\alpha^2)  \;.
\end{align*}
The last inequality follows because
%Arguing as in Lemma~\ref{xi:finite}, we see that 
$\EE_x\left[16\,\ell(M_{n_0})^2\,\exp(  \alpha\,\ell(M_{n_0}))\right]$ is finite and uniformly bounded in $x$  and $0<\alpha\leq 1$ by the constant
$16\,n_0^2\,(\log C)^2\, C^{n_0\,\alpha}$.

 Taking $\alpha>0$ sufficiently small the right-hand-side above is less than $1$, which shows that $\xim{n_0}(A)<1$.
Hence, we can choose $0<\alpha_1<\frac{\alpha_0}{2}$ and $0<\rho<1$ such that for all
$\alpha_1\leq \alpha\leq \alpha_0$,  $\xim{n_0}(A)\leq \rho$.

Next, we extend this inequality to all cocycles $B\in\V$. 

Pick  $\rho'\in ]\rho, 1[$ and choose $\delta>0$ such that $\alpha_0\, {n_0}\, C^{12 {n_0}-1}\,\delta < \rho'-\rho$.
Make the neighborhood $\V$ small enough so that $\norm{A-B}_\infty<\delta$ for all $B\in\V$.
Then, using the modulus of continuity for $\xim{n}(B)$,
 for all $B\in\V$ and $\alpha_1\leq \alpha\leq \alpha_0$,
$$ \abs{ \xim{n_0}(A) -  \xim{n_0}(B) }<  \rho'-\rho  \;,$$
which implies
$$ \xim{n_0}(B)\leq \xim{n_0}(A) + \abs{ \xim{n_0}(A) -  \xim{n_0}(B) }  < \rho' \;.$$
By Lemma~\ref{xi:finite}, 
$\xim{j}(B)\leq C$ for all $B\in\V$, $0<\alpha\leq \frac{1}{4\,n_0}$ and $0\leq j\leq {n_0}$.
Shrinking if necessary the constants $\alpha_1$ and $\alpha_0$ above, we may assume that $\alpha_0\leq \frac{1}{4\,n_0}$. 
Thus, because the sequence $\{\xim{n}(B)\}_{n\geq 0}$ is sub-multiplicative, 
letting $\sigma=(\rho')^{1/{n_0}}$ we have $\xim{n}(B)\leq C\,\sigma^n$
for all $B\in\V$, $n\in\N$ and $\alpha_1\leq \alpha\leq \alpha_0$. The proposition follows then from the inequality proven in Lemma~\ref{Va:QKa}.
\end{proof}
 
%(A2)

Next proposition  implies (A2).

\begin{proposition} \label{QCS}
Given $A\in \randcocycles{\infty}{m}$ such that $(K_A,\mu_A,\xi_A)\in\Xfrk$,  there exist a neighborhood $\V$  of $A$ in $\randcocycles{\infty}{m}$, a range  $0<\alpha_1<\frac{\alpha_0}{2} <\alpha_0\leq 1$  and there are  constants $C>0$ and $0<\sigma<1$ such that    
for all  $B\in\V$, $\alpha\in [\alpha_1,\alpha_0]$ and
 $f\in \Hscr_\alpha(\Sigma\times\Proj)$,
$$ \norm{ Q_{B}^n f -\langle f, \mu_B \rangle\, \one }_\alpha \leq C\, \sigma^ n \norm{f}_\alpha\;. $$
\end{proposition}

\newcommand{\CC}{{\rm C}}

\begin{proof} The argument below is an adaptation of the proof of Theorem 3.7 in~\cite{Bou}.

Take the neighbourhood $\V$, and the constants $\alpha_0>0$, $C>0$ and $0<\sigma<1$ 
given by Proposition~\ref{contr:B}.
Enlarging the constants $C>0$ and  $0<\sigma<1$ 
we can assume that the conditions of definition~\ref{Strong:mixing}
are also satisfied with $\rho=\sigma$. By Lemma~\ref{Va:QKa}, given $B\in \V$ and any $K_B$-stationary measure $\nu_B$,
$$
v_\alpha(  Q_{B}^n f -\langle f, \nu_B\rangle\, \one )  = v_\alpha( Q_{B}^n f ) 
 = v_\alpha(f) \, \xim{n}(B) \leq C\,\sigma^ n\,\norm{f}_\alpha \;.
$$
Hence it is now enough to prove that
$$ \norm{ Q_{B}^n f -\langle f, \nu_B \rangle\, \one }_\infty \leq C\, \sigma^ n \norm{f}_\alpha\;. $$
%The measure $\nu_B$ does not play any role in the previous argument.

We define four families of transformations
$$T^{(i)}_{B,n,m}:L^\infty(\Sigma\times\Proj)\to L^\infty(\Sigma\times\Proj)\quad  i=0,1,2, 3\;, $$
depending on $B\in\V$, and $n\geq m$, $n,m\in\N$, which act continuously on the scale of Banach spaces
$\Hscr_\alpha(\Sigma\times\Proj)$ with $0<\alpha\leq \alpha_0$.

\blob \; $(T^{(0)}_{B,n} f)(x,p) := (Q_{B}^ n f)(x,p) = \EE_x\left[ f(e_n, B^{(n)}\,p ) \right] $.

\blob \; $(T^{(1)}_{B,n,m} f)(x,p) := \EE_x\left[ f(e_n, (B^{(m)}\circ \transl^{n-m})\,p ) \right] $.

\blob \; $(T^{(2)}_{B,m} f)(x,p) :=  \EE_\mu\left[ f(e_m, B^{(m)}\,p ) \right] $.\\
$T^{(2)}_{B,m}$ maps $\Hscr_\alpha(\Sigma\times\Proj)$ onto the space $\Hscr_\alpha(\Proj)$
of $\alpha$-H\"older continuous functions, constant in $x$.
In particular $T^{(2)}_{B,m}:\Hscr_\alpha(\Sigma\times\Proj)\to \CC(\Proj)$
is a compact  transformation.

\blob \; $(T^{(3)}_{B} f)(x,p) :=  \int f\, d\nu_B $, where $\nu_B$ is any $K_B$-stationary measure.\\
$T^{(3)}_{B}$ maps $L^ \infty(\Sigma\times\Proj)$ onto the space of constant functions.
In particular the linear transformation $T^{(3)}_{B}:\Hscr_\alpha(\Sigma\times\Proj)\to \CC(\Proj)$ has rank $1$.

We claim that for all $B\in\V$ and all $f\in\Hscr_\alpha(\Sigma\times\Proj)$ with
$0<\alpha\leq \alpha_0$, 
 for all $n,m\in\N$ with $n\geq m$, and all $(x,q)\in\Sigma\times\Proj$,

\begin{enumerate}
\item[(1)] \;
$\displaystyle \abs{ (T^{(0)}_{B,n} f)(x,q) -  (T^{(1)}_{B,n,m} f)(x,q)}  \leq C\,\sigma^m \norm{f}_\alpha$.

\item[(2)]  \;
$\displaystyle \abs{ (T^{(1)}_{B,n,m} f)(x,q) -  (T^{(2)}_{B,m} f)(q)}  \leq C\,\sigma^{n-m} \norm{f}_\alpha$.

\item[(3)] \;
$\displaystyle \abs{ (T^{(2)}_{B,m} f)(q) -  (T^{(2)}_{B,n} f)(q)}  \leq C\,\sigma^m \norm{f}_\alpha$.
\end{enumerate}
 
We will conclude the proposition  before proving these three claims.
Setting $n=2m$ in (1) and (2), and $n=\ell$ in (3),
 for all $\ell \geq m$,   $B\in\V$ and  $f\in\Hscr_\alpha(\Sigma\times\Proj)$ with
$0<\alpha\leq \alpha_0$,
\begin{equation}\label{QT2}
 \norm{ Q_{B}^{2m} f -  T^{(2)}_{B,\ell} f}_\infty \leq 3\,C\,\sigma^{m} \norm{f}_\alpha\;.
\end{equation}
The sequence $\{ T^{(2)}_{B,\ell} f\}_{\ell\geq 0}$ is relatively compact in $\CC(\Proj)$.
Hence the set $S_f$ of its sublimits in $\left(\CC(\Proj), \norm{\cdot}_\infty\right)$
is non-empty. Take any $g\in S_f$ and any $K_B$-stationary probability measure $\nu_B$.
We claim that $g=\langle f, \nu_B\rangle\,\one$.

From~ \eqref{QT2} we have for all $m\in\N$,
$$ \norm{ Q_{B}^{2m} f -  g}_\infty \leq 3\,C\,\sigma^m \norm{f}_\alpha\;.$$
On the other hand, since $v_\alpha(Q_{B}^{2m} f)\leq C\,\sigma^{2m}\,\norm{f}_\alpha$, we get $v_\alpha(g)=0$,
which implies that $g$ is constant. But  
$\langle Q_{B}^{2m} f , \nu_B \rangle = \langle  f , (Q_{B}^{2m})^ \ast \nu_B \rangle 
= \langle  f ,  \nu_B \rangle$  implies that
$\langle g, \nu_B \rangle = \langle  f ,  \nu_B \rangle$. Therefore  $g=\langle  f ,  \nu_B \rangle\,\one$, and also
$$ \norm{ Q_{B}^{2m} f -  \langle  f ,  \nu_B \rangle\,\one}_\infty \leq 3\,C\,\sigma^m \norm{f}_\alpha\qquad \forall\, m\in\N\;.$$
This concludes the proof.

To finish we still have to prove the three claims:

\noindent
Claim (1): Denote by $\Fcal_n$ the sub $\sigma$-field  generated by
the random variables $e_1,\ldots, e_n$. Note that for any random variable $f:X\to \C$
$$ \EE_x(f)=\EE_x \left\{ \EE_{e_n}(f\vert \Fcal_n) \right\} \;. $$
Then, using this fact we have,
\begin{align*}
& \abs{(T^{(0)}_{B,n} f - T^{(1)}_{B,n,m} f)(x,q)}  =\abs{ \EE_x\left[f(e_n,B^{(n)} q)  -  f(e_n, (B^{(m)}\circ \transl^{n-m})\,q )\right] }\\
&\qquad \leq \EE_x\left[\abs{ f(e_n,(B^{(m)}\circ\transl^{n-m}) B^{(n-m)}  q) - f(e_n, (B^{(m)}\circ \transl^{n-m})\,q ) } \right] \\
&\qquad \leq \norm{f}_\alpha \,\EE_x\left[ \delta\left( 
(B^{(m)}\circ\transl^{n-m}) B^{(n-m)} \, q, 
(B^{(m)}\circ \transl^{n-m})\,q
 \right)^ \alpha \right] \\
 &\qquad \leq \norm{f}_\alpha \,\EE_x\left\{ 
 \EE_{e_{n-m}}\left[ \delta\left( 
(B^{(m)}\circ\transl^{n-m}) B^{(n-m)} \, q, 
(B^{(m)}\circ \transl^{n-m})\,q
 \right)^ \alpha \vert \Fcal_{n-m}\right] \right\} \\
 &\qquad \leq \norm{f}_\alpha \, \sup_{x,p,q}  \EE_x\left[  \delta\left( 
B^{(m)}  p, 
B^{(m)} q
 \right)^\alpha \right] \\
 &\qquad \leq \norm{f}_\alpha \, \sup_{x,p\neq q}  \EE_x\left[ \left( \frac{ \delta( 
B^{(m)}  p, 
B^{(m)} q ) }{\delta(p,q)} \right)^\alpha \right] = \norm{f}_\alpha \, \xim{m}(B)\leq
\norm{f}_\alpha \, C\,\sigma^ m\;.
\end{align*}

\noindent
Claim (2): Defining $\varphi_{m,q}(x):=\EE_x\left[ f(e_m, B^{(m)} q) \right]$,
because $(K,\mu)$ is strongly mixing on $L^ \infty(\Sigma)$ we have
\begin{align*}
& \abs{ (T^{(1)}_{B,n,m} f  -  T^{(2)}_{B,m} f)(x, q)}   = 
\abs{ \EE_x\left[ f(e_n, (B^{(m)}\circ \transl^{n-m})\,q ) \right] - 
\EE_\mu\left[ f(e_m, B^{(m)}\,q ) \right] } \\
& \qquad = \abs{ \EE_x\left\{ \EE_{e_{n-m}}\left[ f(e_n, (B^{(m)}\circ \transl^{n-m})\,q )\right] \right\} - 
\EE_\mu\left[ f(e_m, B^{(m)}\,q ) \right] } \\
& \qquad = \abs{ \EE_x\left[ \varphi_{m,q}(e_{n-m}) \right] - 
\int  \varphi_{m,q} \,d\mu   }  = \abs{ Q_K^{n-m} \varphi_{m,q} -\int  \varphi_{m,q} \,d\mu \,\one}
\leq C\,\sigma^ {n-m}\;.
\end{align*}

\noindent
Claim (3): Because $\mu$ is $K$-stationary,
\begin{align*}
& \abs{ (T^{(2)}_{B,m} f)(q) -  (T^{(2)}_{B,n} f)(q)}  = 
\abs{ \EE_\mu\left[ f(e_m, B^{(m)}\,q ) \right] - \EE_\mu\left[ f(e_n, B^{(n)}\,q ) \right] }\\
&\qquad = 
\abs{ \EE_\mu\left[ f(e_m, B^{(m)}\,q ) \right] - \EE_\mu\left[ f(e_m, B^{(m)}\,B^{(n-m)}\, q ) \right] }\\
&\qquad \leq 
\EE_\mu\left[ \abs{ f(e_m, B^{(m)}\,q )  -  f(e_m, B^{(m)}\,B^{(n-m)}\, q )} \right]  \\
&\qquad \leq 
\norm{f}_\alpha\, \EE_\mu\left[ \delta( B^{(m)}\,q, B^{(m)}\,B^{(n-m)}\, q )^\alpha \right]  \\
&\qquad \leq 
\norm{f}_\alpha\, \EE_\mu\left\{ 
\EE_{e_{n-m}}\left[
\delta( B^{(m)}\,q, B^{(m)}\,B^{(n-m)}\, q )^\alpha \right] \right\}  \\
&\qquad \leq 
\norm{f}_\alpha\, \sup_{x, p,q}  
\EE_{x}\left[
\delta( B^{(m)}\,q, B^{(m)}\,p )^\alpha \right]   \\
 &\qquad \leq \norm{f}_\alpha \, \sup_{x,p\neq q}  \EE_x\left[ \left( \frac{ \delta( 
B^{(m)}  q, 
B^{(m)} p ) }{\delta(q,p)} \right)^\alpha \right] = \norm{f}_\alpha \, \xim{m}(B)\leq
\norm{f}_\alpha \, C\,\sigma^ m\;.
\end{align*}
\end{proof}

\begin{corollary}
\label{uniqueness of stationary measure}
Given $A\in \randcocycles{\infty}{m}$ such that $(K_A,\mu_A,\xi_A)\in\Xfrk$, the kernel $K_A$ on the product space $\Sigma\times\Sigma\times\Proj$
has a unique  stationary measure.
\end{corollary}

\begin{proof}
In the proof of Proposition~\ref{QCS}  we have shown that given a function
$f\in \Hscr_\alpha(\Sigma\times\Proj)$, if we denote by $S_f$
the set of sublimits of $\{T^{(2)}_{B,\ell} f\}_{\ell\geq 0}$,
then $S_f=\{ \langle f, \nu_B\rangle\,\one\}$ for any $K_B$-stationary measure $\nu_B$.

Hence, given any other $K_B$-stationary measure $\mu_B$,
and $f\in \Hscr_\alpha(\Sigma\times\Proj)$, we have
 $\langle f, \nu_B\rangle = \langle f, \mu_B\rangle$.
Since
$\Hscr_\alpha(\Sigma\times\Proj)$
is dense in 
$L^\infty(\Sigma\times\Proj)$, it follows that
$\nu_B=\mu_B$.
\end{proof}

The Laplace-Markov operator $Q_{A,z}$ of the observed Markov system $(K_A,\mu_A,\xi_A)$ is given by
\begin{equation}
\label{fiber Laplace-Markov operator}
 (Q_{A,z} f)(x,y,p)=
\int_\Sigma f(y,z, A(y,z)\,p)\,\norm{A(y,z)}^z\,
K(y,dz)\;. 
\end{equation}

Like the Markov operator $Q_A$ defined in~\eqref{QA}, the Laplace-Markov operator $Q_{A,z}$ leaves invariant the
 subspaces $ \Hscr_\alpha(\Sigma\times\Proj)$,
for all small enough  $\alpha>0$.
Choose $0<\alpha_1<\alpha_0\leq 1$ according to proposition~\ref{contr:B}.

Assumption (A3) is automatically  satisfied because $\norm {A}_\infty< \infty$ and $\norm {A^{-1}}_\infty< \infty$ which imply that
$\xi_A\in \Hscr_\alpha(\Sigma\times\Proj)$ for all $\alpha>0$. Note that $Q_{A,z}=Q_A\circ D_{e^{z\,\xi_a}}$, where $D_{e^{z\,\xi_a}}$ denotes the multiplication operator by $e^{z\,\xi_a}$. 
This is a bounded operator because $\Hscr_\alpha(\Sigma\times\Proj)$ is a Banach algebra containing the function $e^{z\,\xi_a}$.

%(A4)
Finally the next lemma proves (A4).
\begin{lemma}  \label{C2:lema}
Given $A,B\in\randcocycles{\infty}{m}$ and $b>0$,  there is a constant $C_2>0$  such that for all $f\in \Hscr_\alpha(\Sigma\times\Proj)$,
and all $z\in\C$ such that ${\rm Re}  \, z\leq b$,
$$ \norm{ Q_{A,z} f - Q_{B,z} f }_\infty \leq C_2\, d_\infty(A,B)^ \alpha\,\norm{f}_\alpha\;.$$
\end{lemma}

\begin{proof}  
A simple computation shows that  for all $z\in\C$ with ${\rm Re}\, z\leq b$, and all $A,B\in\GL(d,\R)$,
$$ \abs{\norm{A\,p}^z - \norm{B\,p}^z }\leq b\, \max\{\norm{A}^{b-1},\norm{B}^ {b-1}\}\,\norm{A-B}\;. $$
Hence
\begin{align*}
& \abs{ (Q_{A,z} f - Q_{B,z} f)(x,p)}  \leq
\EE_x\left[ \abs{ \norm{A\,p}^z\,f(e_1,A\,p) -
\norm{B\,p}^z\,f(e_1,B\,p)  } \right] \\
&\qquad \leq \norm{f}_\infty\, \EE_x\left[ \abs{ \norm{A\,p}^z - \norm{B\,p}^z  } \right] +   \norm{B}_\infty^b\, \EE_x\left[ \abs{  f(e_1,A\,p) -
 f(e_1,B\,p)  } \right] \\
&\qquad \leq   b\, \max\{\norm{A}_\infty^{b-1},\norm{B}_\infty^ {b-1}\}\,\norm{A-B}_\infty\, \norm{f}_\infty  +   \norm{B}_\infty^b\, v_\alpha(f)\, \EE_x\left[ \delta( A\,p ,B\,p)^ \alpha \right] \\
&\qquad \leq   b\, \max\{\norm{A}_\infty^{b-1},\norm{B}_\infty^ {b-1}\}\,\norm{A-B}_\infty\, \norm{f}_\infty\\
&\qquad \qquad + \; \norm{B}_\infty^b\, v_\alpha(f)\, \norm{A-B}_\infty^\alpha \;
 \leq  \; C_2\, \norm{f}_\alpha\, d_\infty(A,B)^ \alpha \;,
\end{align*}
where $C_2=\max\{\norm{B}_\infty^b, b\,\norm{B}_\infty^{b-1}, b\,\norm{A}_\infty^{b-1}\}$.
\end{proof}

% proof of fiber LDT theorem

\begin{proof}[Proof of Theorem~\ref{Fiber:LDT}]
The space of observed Markov systems
$\Xfrk$  satisfies all assumptions (A1)-(A4). Hence,  by Theorem~\ref{ALDE:uniform}, there exists a  neighborhood $\V$ of 
$(K_A,\mu_A,\xi_A)\in\Xfrk$, which we identify with a neighborhood of
$A\in \randcocycles{\infty}{m}$, and there are constants $\varepsilon_0, C, \ctwo>0$ such that
for all $B\in\V$,   $0<\varepsilon<\varepsilon_0$,  
$(x,p)\in\Sigma\times \Proj$  and  $n\in\N$,
$$ \Pp_x\left[ \, \abs{\frac{1}{n}\,\log \norm{B^{(n)}\,p }- L_1(B,\mu) } \geq \varepsilon \,\right]\leq C\,e^{ - {\varepsilon^2\over {2\,\ctwo}}\, n  } \;. $$
Integrating w.r.t. $\mu$ we get for all $p\in\Proj$,
$$ \Pp_\mu\left[ \, \abs{\frac{1}{n}\,\log \norm{B^{(n)}\,p }- L_1(B,\mu) } \geq \varepsilon \,\right]\leq C\,e^{ - {\varepsilon^2\over {2\,\ctwo}}\, n  } \;. $$
Choose the canonical basis $\{e_1,\ldots, e_d\}$ of $\R^m$ and consider the following norm $\norm{\cdot}'$ on the space of matrices  $\Mat_d(\R)$,
$\norm{M}':= \max_{1\leq j\leq d} \norm{M\, e_j}$.
Since this norm is equivalent to the operator norm, for all $B\in\V$,  $p\in\Proj$ and $n\in\N$,
$$ \norm{ B^{(n)}\,p} \leq \norm{ B^{(n)} } \lesssim \norm{ B^{(n)} }' =   \max_{1\leq j\leq d} \norm{ B^{(n)}\,e_j}\;. $$
Thus a simple comparison of the deviation sets gives
$$ \Pp_\mu\left[ \, \abs{\frac{1}{n}\,\log \norm{B^{(n)} }- L_1(B,\mu) } \geq \varepsilon \,\right]\lesssim \,e^{ - {\varepsilon^2\over {2\,\ctwo}}\, n  }  $$
for all $B\in\V$,   $0<\varepsilon<\varepsilon_0$ and  $n\in\N$.
\end{proof}

\bigskip

\section{Deriving continuity of the Lyapunov exponents}
\label{random_continuity}

In this last section we use the LDT estimates (theorems~\ref{Base:LDT} and~\ref{Fiber:LDT}) to derive the continuity of the Lyapunov exponents and of the  Oseledets's filtration / decomposition. We give some simple generalizations of the continuity results and explain the method's limitations regarding the continuity of the LE in the reducible case.

\subsection{Proof of the continuity}
\label{random_cont_proof}

\begin{proof}[Proof of Theorem~\ref {theorem: continuity of LE}]
 Let $(K,\mu)$ be a strongly mixing Markov system, and  consider the associated Markov shift $(X,\Pp_\mu,T)$.

The collection $\mathcal{C}=\{(\irredcocycles{\infty}{m},d_\infty)\}_{m\in\N}$ is a space of measurable cocycles in the sense of~\cite[Definition 1.8]{LEbook} (see also~\cite[Definition 1.1]{LEbook-chap3}).
We are going to apply the abstract continuity theorem (ACT)~\cite[Theorem 1.6]{LEbook} (see also~(\cite[Theorem 1.1]{LEbook-chap3} and~\cite[Theorems 3.2 and 3.3]{LEbook-chap4}) to this space   of totally irreducible cocycles   
over $(X,\Pp_\mu,T)$.

Consider the space of LDT parameters
$\params=\N\times \devfs\times \mesfs$, where $\devfs$ is the set of constant deviation functions $\devf (t) \equiv \ep$,  $0 < \ep <1$, and we use the set of exponential functions
$\mesfs=\{\, \mesf (t) \equiv M\,e^{- c \, t}\,\colon\,
M<\infty, \, c>0\,\}$ to measure the deviation sets.

Define $\observables$ to be the set of observables $\xi:X\to \R$ which depend only on finitely many coordinates.
Finally, take $p=\infty$.

We now check the four assumptions of the ACT.

1.\; The set $\observables$ is compatible with all cocycles $A\in \randcocycles{\infty}{m}$, because for any
set $F\in \mathscr{F}_N(A)$ its indicator function $\ind_{F}$ depends only on finitely many coordinates, i.e., $\ind_{F}\in\observables$.

2.\;  Given an  observable 
$\xi\in\observables$ there exists $p\in\N$ such that $\xi\circ T^p$ depends only on negative coordinates, i.e., coordinates $x_j$ with $-p\leq j\leq 0$. This implies that $\xi\circ T^p\in \Hscr_\alpha(X^-)$. By Theorem~\ref{Base:LDT}, the observable
$\xi\circ T^p$ satisfies a base-LDT estimate w.r.t. $\params$. Since 
$\abs{S_n(\xi)-S_n(\xi\circ T^p)}$ converges uniformly to zero  as $n\to\infty$, it follows that
$\xi$ satisfies  base-LDT estimates too.

3.\;   The $L^p$-boundedness assumption is automatic because  $p=\infty$ and the functions $A$ and $A^{-1}$ are bounded.

4.\; Given $A\in \irredcocycles{\infty}{m}$ such that $L_1(A)>L_2(A)$, by Theorem~\ref{Fiber:LDT} the cocycle $A$ satisfies uniform fiber-LDT estimates w.r.t. $\params$.

A simple computation shows that the modulus of continuity associated to the choice of deviation function sets $\devfs$ and $\mesfs$ above
corresponds to H\"older continuity.
Hence, this theorem follows from the conclusions of the ACT.
\end{proof}

\subsection{Some generalizations}
\label{random_generalizations}

Consider a compact metric space $\Sigma$.

A {\em Markov kernel} of order $p\in\N$ on $\Sigma$ is a map
$K:\Sigma^p\to {\rm Prob}(\Sigma)$
that assigns a probability measure $K(x_0,\ldots, x_{p-1}, dy)$ on $\Sigma$ to each tuple $(x_0,\ldots, x_{p-1})\in\Sigma^p$. The concept of Markov kernel in Definition~\ref{Markov:kernel} corresponds to a Markov kernel of order $p=1$.

Any Markov kernel $K$ of order $p$ on $\Sigma$ determines the following Markov kernel $\hat K$ of order $1$ on the product space $\Sigma^p$,
$$ \hat K(x_0,\ldots, x_{p-1}):=
\int_\Sigma  \delta_{(x_1,\ldots, x_p)}\,
K(x_0,\ldots, x_{p-1}, d x_p) \;. $$

A probability measure $\mu$ on $\Sigma^p$ is said to be {\em $K$-stationary} when it is $\hat K$-stationary.
We call a Markov system of order $p$ any  pair $(K,\mu)$, where $K$ is a Markov kernel of order $p$ on $\Sigma$, and $\mu$ is a $K$-stationary probability on $\Sigma^p$. We say that $(K,\mu)$ is strongly irreducible when $(\hat K,\mu)$ is a strongly irreducible
Markov system on $\Sigma^p$.

Given a  Markov system $(K,\mu)$ 
of order $p$, let $\hat\Pp_\mu$  denote the Kolmogorov extension of $(K,\mu)$ on the space of sequences $\hat X:=(\Sigma^p)^\Z$.
Then, letting $\hat T:{\hat X}\to{\hat X}$ denote  the shift homeomorphism, the triple 
 $\left( {\hat X},\hat \Pp_\mu, \hat T\right)$ is a Markov shift.
 
Let $X:= \Sigma^\Z$ and consider the maps

\blob\; $\psi:X\to {\hat X}$,
$\psi\{x_n\}_{n\in\Z}=\{(x_{n},\ldots, x_{n+p-1})\}_{n\in\Z}$,

\blob\; $\pi:{\hat X}\to X$,
$\pi\{(x_{0,n},\ldots, x_{p-1,n})\}_{n\in\Z}=\{x_{0,n}\}_{n\in\Z}$,

\noindent
which satisfy $\pi\circ \psi={\rm id}_X$.

Defining $\Pp_\mu:= \pi_\ast \hat\Pp_\mu$,
these maps are bimeasurable isomorphisms conjugating the shifts on $\left( {\hat X},\hat\Pp_\mu\right)$ and $(X,\Pp_\mu)$, where the measure $\Pp_\mu$ is invariant under the shift $T:X\to X$.
The triple $(X,\Pp_\mu,T)$ is called a Markov shift of order $p$.

Consider now the space $\prandcocycles{\infty}{m}$ of measurable functions 
$A:X\to \GL(m,\R)$ which depend  only on the coordinates $(x_0,\ldots, x_p)\in\Sigma^{p+1}$
with  $\norm{A}_\infty<\infty$ and $\norm{A^{-1}}_\infty<\infty$.
Note that the iterates of $A$ are
$$ \An{n}(x) = A(x_{n-1},\ldots, x_{n-1+p})\,\ldots \, A(x_{1},\ldots, x_{1+p})\, A(x_{0},\ldots, x_{p}) \;.$$

We identify $\prandcocycles{\infty}{m}$ as a space of functions 
$A:\Sigma^{p+1}\to\GL(m,\R)$. Each such function  determines a locally constant cocycle  over the Markov shift $(X,\Pp_\mu,T)$.

Given $A\in \prandcocycles{\infty}{m}$,
we define $\hat A: \Sigma^p\times \Sigma^p\to \GL(m,\R)$ 
$$ \hat A \left( (x_0,\ldots, x_{p-1}), (y_0,\ldots, y_{p-1}) \right):=  A(x_0,\ldots, x_{p-1}, y_{p-1})\;.$$
Identifying $\hat A$ with a function 
$\hat A:\hat X \to\GL(m,\R)$
we have  $\hat A \circ \psi = A$.
Hence the cocycles $(\hat T,\hat A)$ and $(T,A)$ are conjugated.

The cocycle $(T,A)$ over the Markov shift $(X,\Pp_\mu,T)$ will be called a {\em random Markov cocycle} of {\em order} $p$.

Define $\pirredcocycles{\infty}{m}$ to be the subspace of {\em totally irreducible} cocycles $A\in \prandcocycles{\infty}{m}$, i.e., the subspace of cocycles $A$   such that
$\hat A$  is  totally irreducible over $(\hat X,\hat\Pp_\mu,\hat T)$.

From these considerations and Theorem~\ref{theorem: continuity of LE} we obtain the following result.

\begin{theorem}
\label{thm: cont of LE p-Markov}
Let  $(K,\mu)$ be a  strongly mixing  Markov system of order $p\in\N$. 

Then all Lyapunov exponents $L_j:\pirredcocycles{\infty}{m}\to \R$, with $1\leq j \leq m$, 
the Ose\-le\-dets filtration $\filt:\pirredcocycles{\infty}{m}\to \Filt(X,\R^m)$,
and the Oseledets decomposition  $\dec:\pirredcocycles{\infty}{m}\to \Dec(X,\R^m)$,
are continuous functions of the cocycle
$A\in \pirredcocycles{\infty}{m}$. 

Moreover,  if $A\in \pirredcocycles{\infty}{m}$ has a $\tau$-gap pattern then the functions
$\Lambda^\tau$, $\filt^\tau$ and $\dec^\tau$ are
H\"older continuous in a neighborhood of $A$.
\end{theorem}

\medskip

In particular, all conclusions above  on the continuity of the  LE, the Oseledets filtration, and the Oseledets decomposition, apply to irreducible and locally constant cocycles over strongly mixing Markov and Bernoulli shifts.

\medskip

The abstract setting developed in section~\ref{random_as}
is general enough to deal with cocycle having singularities, i.e., points $x\in X$ where the matrix $A(x)$ is singular.
Consider the family of spaces $\randcocycles{a}{m}$, with $0<a<\infty$, consisting of all bounded measurable functions $A:\Sigma\times\Sigma \to \GL(m,\R)$ such that for some  $C>0$  and all $x\in\Sigma$, 
\begin{equation*} 
\eta_A^a(x):=\int_{\Sigma} \norm{A(x,y)^{-1}}^a \, K(x,dy)\leq C\;.
\end{equation*}
Equip this space with the distance
$$ d_a(A,B):=   \norm{A-B}_\infty + \norm{\eta_A^a-\eta_B^a}_\infty \;. $$
 
The collection $\mathcal{C}=\{ (\randcocycles{a}{m},d_a)\}_{m\in\N}$ is not  a space of measurable cocycles,
because (2) of~\cite[Definition 1.8]{LEbook} (see also~\cite[Definition 1.1]{LEbook-chap3}) fails. However, both the uniform fiber-LDT estimates and the continuity
statments about the LE can be extended to the spaces $\irredcocycles{a}{m}$ of totally irreducible cocycles in $\randcocycles{a}{m}$. More precisely, it can be proved that Theorem~\ref{Fiber:LDT} holds  
for all $a\geq 4$, and 
Theorem~\ref {theorem: continuity of LE} holds 
for all $a\geq 4\,m$.

\subsection{Method limitations}
\label{random_limitations}
We need the irreducibility assumption in order to prove uniform fiber LDT estimates in Theorem~\ref{Fiber:LDT}. The proof exploits the fact that for irreducible  cocycles there is some Banach algebra of measurable functions, {\em independent} of the cocycle, where the associated Laplace-Markov operators act  as quasi-compact and simple operators (see Proposition~\ref{QCS}). For reducible cocycles this fact may still be  true, and it could eventually lead  to fiber LDT estimates. However, the Banach algebra would have to be tailored to the cocycle, and hence the scheme of proof presented here would not provide the required uniformity.

\bigskip

\subsection*{Acknowledgments}

Both authors would like to acknowledge fruitful conversations they had with A. Baraviera, G. Del Magno and  J. P. Gaiv\~ao about large deviation estimates for random cocycles.

The first author was supported by 
Funda\c{c}\~{a}o  para a  Ci\^{e}ncia e a Tecnologia, 
UID/MAT/04561/2013.

The second author was supported by the Norwegian Research Council project no. 213638, "Discrete Models in Mathematical Analysis".  He is also grateful to the University of Lisbon and to the Institut Mittag-Leffler (through its "Research in Peace'' program) for their hospitality during the summer of 2013, when this project began.

\bigskip

\bibliographystyle{amsplain} 
%\bibliography{references} 
\providecommand{\bysame}{\leavevmode\hbox to3em{\hrulefill}\thinspace}
\providecommand{\MR}{\relax\ifhmode\unskip\space\fi MR }
% \MRhref is called by the amsart/book/proc definition of \MR.
\providecommand{\MRhref}[2]{%
  \href{http://www.ams.org/mathscinet-getitem?mr=#1}{#2}
}
\providecommand{\href}[2]{#2}

\end{document}